\documentclass{article}
\usepackage{amssymb}
\usepackage{amsmath}
\usepackage{amsthm}
\usepackage{color}
\usepackage{graphicx}
\theoremstyle{plain}
\newtheorem{theorem}{Theorem}[section]
\newtheorem{lemma}{Lemma}[section]
\newtheorem{corollary}{Corollary}[section]
\newtheorem{proposition}{Proposition}[section]

\theoremstyle{definition}
\newtheorem{definition}{Definition}[section]
\newtheorem{example}{Example}[section]

\theoremstyle{remark}
\newtheorem{remark}{Remark}[section]

\usepackage{lineno}

\begin{document}
\begin{titlepage}
\title{Explicit minimisation of a convex quadratic under a general quadratic constraint: a global, analytic approach}

\author{%
Casper J. Albers \\ Department of Psychometrics and Statistics, \\ University of Groningen, \\ Grote Kruisstraat 2/1, \\ 9712TS Groningen, \\ The Netherlands.  \\ c.j.albers@rug.nl \\ www.casperalbers.nl \\ corresponding author
\and
Frank Critchley \quad John Gower \\Department of Mathematics and Statistics, \\ The Open University, \\ Walton Hall, \\ Milton Keynes MK7 6AA, \\ United Kingdom \\ f.critchley@open.ac.uk \quad john.gower@open.ac.uk
}
\maketitle
\begin{abstract}
A novel approach is introduced to a very widely occurring problem, providing a
complete, explicit resolution of it: minimisation of a convex quadratic under a
general quadratic, equality or inequality, constraint. Completeness comes via
identification of a set of mutually exclusive and exhaustive special cases.
Explicitness, via algebraic expressions for each solution set. Throughout,
underlying geometry illuminates and informs algebraic development. In
particular, centrally to this new approach, affine equivalence is exploited to
re-express the same problem in simpler coordinate systems. Overall, the
analysis presented provides insight into the diverse forms taken both by the
problem itself and its solution set, showing how each may be intrinsically
unstable. Comparisons of this global, analytic approach with the,
intrinsically complementary, local, computational approach of (generalised)
trust region methods point to potential synergies between them. Points of
contact with simultaneous diagonalisation results are noted.


\end{abstract}

\end{titlepage}

\section{Introduction} \label{SECTION: Introduction}

\subsection{Background}
This paper introduces a novel approach to a very widely occurring problem:
minimisation of a convex quadratic under a general quadratic, equality or
inequality, constraint. This may arise by itself, or as a component of a
larger problem -- notably, as a single iteration in minimising a smooth convex
objective function under a smooth constraint. In either case, we treat solving
the above problem as of interest \textit{in itself}.

Statistical instances of this problem -- our primary motivation -- occur, for
example, in minimum distance estimation, Bayesian decision theory, generalised
linear models with smooth constraints, canonical variate analysis with fewer
samples than variables, various forms of oblique Procrustes analysis,
Fisher/Guttmann estimation of optimal scores, spline fitting, estimation of
Hardy-Weinberg equilibrium, and size and location constraints in iterative
missing value procedures in Procrustes analysis. Associated literature is
reviewed, united and extended in Albers et al. \cite{AlbersJMVA1}. Some of these instances
are elaborated in Albers et al. \cite{AlbersJMVA2,AlbersCanonical}.

Within the optimisation and numerical analysis literatures, the above problem
has particularly close connections with (generalised) trust region ((G)TR)
methods, being a special case of the GTRS problem defined in \cite{PongWolkowicz}. At the same
time, the global, analytic approach presented here is, intrinsically,
\textit{complementary} to the local, computational one in the GTR literature.
Valuable in itself, this new approach points to potential synergies with
existing methodology, as we briefly indicate in closing.

In the main, GTR methods adapt a nonlinear optimisation protocol to address
the problem on hand. Regularity conditions (constraint qualifications) are
introduced, as required, under which the Karush-Kuhn-Tucker conditions are
necessary for a local optimum, assuming such exists, additional conditions
typically being required to ensure their sufficiency. Algorithms are then
designed to seek a point satisfying \textit{all} of these conditions. The
literature on TR and GTR methods can be accessed via the excellent accounts
presented in \cite{Conn} and \cite{PongWolkowicz,More,SternWolkowicz} respectively. See also \cite{Sorensen,Gander,Elden1982,Xia2015,Wang,TuyHoaiPhuong,Jeyakumar,Jiang},
algorithms to handle the extreme forms of ill-conditioning that can occur
being presented, for example, in \cite{Elden1990,Elden2002}. Recent work in \cite{Adachi}
on non-iterative algorithms extends \cite{Gander1989}, and notes the
possibility of non-unique solutions \cite[cf.][]{Martinez}. Other related
recent work includes \cite{JianLiWu,Ben-Tal2013,Flores-Bazán2013,Locatelli,XingFangEtal}.

In contrast to the GTR approach, the one introduced here exploits a series of
nonsingular affine transformations to re-express the \textit{same} problem in
successively simpler forms, each transformation corresponding to a convenient
change of coordinate system. Solving the last of these \textit{affinely
equivalent}\ forms solves at once the initial problem, via back
transformation. Throughout, underlying geometry illuminates and informs
algebraic development. This approach offers a complete, explicit resolution of
the problem presented at the outset. Completeness comes from analysing
(literally, \textit{splitting}) all of its possible instances into precisely
identified, mutually exclusive and exhaustive, special cases. Explicitness
comes from providing, in each case, closed-form expressions for the solution
set and optimised value attained thereon.

Overall, this new approach highlights the diverse nature of different
instances of both problem and solution set. And, moreover, the possibility of
intrinsic instability -- whereby arbitrarily (hence, undetectably) small changes in problem
specification lead to radical changes in the form of problem and/or solution
set -- pinpointing where this occurs. It may then, in practice, be
\textit{impossible} to be sure which of several forms applies. 

Concerning this
approach, Critchley \cite{CritchleyReport}, cited in \cite{More}, provided an early account of the
strictly convex form, the general case being first addressed in the technical
report  \cite{AlbersCritchleyGowerTechReport}. Again, Gower and Dijksterhuis \cite{GowerDijksterhuis} 
addressed the problem in the context of Procrustes analysis and gave a preliminary algorithm, further 
worked out in Albers and Gower \cite{AlbersProcrustes}.

Finally, in both approaches, there are direct points of contact with
simultaneous diagonalisation results \cite{Newcomb} as detailed below and, for example, in
\cite{Jiang} in the GTR case.

\subsection{Notation and conventions} 
The following notation and conventions are used. Terms involving arrays of
vanishing order are absent. $R^{n}$ is endowed with the standard Euclidean
inner product, inherited by each of its subspaces. Its zero member is denoted by
$0_{n}$, and the span of its first $r$ unit coordinate vectors $S_{r} $.
$M_{n}$\ denotes the set of all $n\times n$\ real symmetric matrices,\ with
zero member $O_{n}$. Subscripts denoting the order of arrays may be omitted
when no confusion is possible. Positive (semi-)definiteness of a matrix $A$
is denoted by $A\succ O$ (respectively, $A\succeq O$), the latter terminology
here implying that $A$ is \textit{singular}. The Moore-Penrose inverse of $A$ is
denoted by $A^{-}$. Finally, $\operatorname*{diag}(\cdot,\ldots,\cdot)$\ denotes a
(block)diagonal matrix with the diagonal entries listed, while $\subset
$\ denotes strict inclusion.

For brevity, straightforward proofs are omitted.

\subsection{The general problem $\mathbb{P}_{\omega}$}

The general equality-constrained problem is as follows.

\begin{definition} \label{DEFINITION: problem}
For $A$, $B$ in $M_{n}$, $t$, $b$ in $R^{n}$\ and $k$ in $R$, $A\neq O_{n}$,
writing $\omega=(A,B,t,b,k)$, the $n$-variable problem $\mathbb{P}_{\omega}
$\ is:%
\begin{align*}
\text{find }\underline{L}_{\omega}  & :=\inf_{x\in R^{n}}L_{\omega
}(x),\;\text{with } L_{\omega}(x):=(x-t)^{\prime}A(x-t),\\
\text{and }\widehat{X}_{\omega}  & :=\{\widehat{x}\in X_{\omega}%
:L_{\omega}(\widehat{x})=\underline{L}_{\omega}\}\\
\text{subject to }Q_{\omega}(x)  & :=x^{\prime}Bx+2b^{\prime}x-k=0\text{,}%
\end{align*}
where the objective (loss) function $L_{\omega}(\cdot)$ is convex and the
feasible set $X_{\omega}:=\{x\in R^{n}:Q_{\omega}(x)=0\}$ nonempty; when the
solution set $\widehat{X}_{\omega}$ is nonempty, $\underline{L}_{\omega}%
=\min\{L_{\omega}(x):x\in X_{\omega}\}$ may be written\ as $\widehat{L}%
_{\omega}$. The set of all such $\omega$\ is denoted by $\Omega_{n}$. Where no
confusion is possible, we may omit the subscript $\omega$. We say that
$\mathbb{P}$\ is (a) \textit{centred} if the target $t=0_{n}$, (b) \textit{a
(partitioned) least-squares problem} if $A$ has the form
$\operatorname*{diag}(1,...1,0,...,0)$, and (c) \textit{a full least-squares problem} if
$A=I$. For any least-squares problem, we partition
\begin{equation}
x=\left(
\begin{array}
[c]{c}%
x_{1}\\
x_{0}%
\end{array}
\right)  ,\;B=\left(
\begin{array}
[c]{ll}%
B_{11} & B_{10}\\
B_{01} & B_{00}%
\end{array}
\right)  ,\;t=\left(
\begin{array}
[c]{c}%
t_{1}\\
t_{0}%
\end{array}
\right)  \text{ \ and \ }b=\left(
\begin{array}
[c]{c}%
b_{1}\\
b_{0}%
\end{array}
\right) \label{DISPLAY: (1,0) partitioning of x, B, t & b}%
\end{equation}
conformably with $A$, subscripts $1$ and $0$ connoting its range and null
spaces, so that $L_{\omega}(x)=\left\Vert x_{1}-t_{1}\right\Vert ^{2}$. In
this way, terms with $0$ subscripts are absent when $A=I_{n}$, so that
$x=x_{1},B=B_{11},t=t_{1}$ and\ $b=b_{1}$.
\end{definition}

\begin{remark}
Denoting the rank of $A$ by $0<r\leq n$, $A\succ O$ and $L(\cdot)$ is strictly
convex when $r=n$, while otherwise $A\succeq O$, with $L(x)=0$ if and only if
$(x-t)$ lies in the $(n-r)$-dimensional null space of $A$. Symmetry of
$A$\ and $B$\ is assumed without loss, while taking $A$ nonzero avoids one
obvious triviality. $X\neq\emptyset$\ avoids another, entailing restrictions
on $(B,b,k)$ characterised (by negation) in Lemma
\ref{RESULT: when does constraint NOT have a solution?} below. The overall sign of
the constraint coefficients can be reversed at will, $\mathbb{P}_{\omega}$
being unchanged under $(B,b,k)\rightarrow(-B,-b,-k)$. Additional consistent
affine constraints do not require separate treatment: they can be substituted
out to arrive at an equivalent instance of $\mathbb{P}_{\omega}$\ in fewer variables.
\end{remark}

\begin{lemma}
\label{RESULT: when does constraint NOT have a solution?}Let $Q(x):=x^{\prime
}Bx+2b^{\prime}x-k$, where $B\in M_{n}$, $b\in R^{n}$\ and $k\in R$. Let
$b$\ have decomposition $b=Bx_{b}+b_{\perp}$, $x_{b}:=B^{-}b$, along the range
and null spaces of $B$, so that $b_{\perp}^{\prime}x_{b}=0$, and let
$k_{+}:=k+b^{\prime}B^{-}b$. Then:
\begin{equation}
Q(x)\equiv(x+x_{b})^{\prime}B(x+x_{b})+2b_{\perp}^{\prime}(x+x_{b}%
)-k_{+}\text{.}\label{DISPLAY: Q:  Lemma 1.1's basic indentity}%
\end{equation}
Accordingly, $Q(x)=0$ does \textbf{not} have a solution if and only if%
\[%
\begin{tabular}[c]{rlll}%
(i): & $[B=O$, & $b_{\perp}=0,$ & $k_{+}\neq0],$\\
or (ii): & $[B$ is nonsingular, & $(-k_{+}B)\succ O],$ & \\
or (iii): & $[B\neq O$ is singular, & $b_{\perp}=0,$ & $k_{+}\neq0,$\\
& $k_{+}B$ has no positive eigenvalues$]$ &  &
\end{tabular}
\]

\end{lemma}

\subsection{Overview of paper \label{SECTION: organisation & overview}}

The solution to $\mathbb{P}_{\omega}$\ when $B=O$\ is well-known. For all
other $\omega\in\Omega$, the infimal value $\underline{L}_{\omega}$\ and
solution set $\widehat{X}_{\omega}$ are given explicitly. The variant of
$\mathbb{P}_{\omega}$ in which the constraint is relaxed to a weak inequality
is also completely resolved. The organisation and principal subsidiary results
of the paper are as follows.

Sections \ref{SECTION: Affine equivalence} and
\ref{SECTION: projected least squares re-formulation}: Theorem
\ref{Proposition: affine change of coordinates} describes \textit{affine
equivalence} on the set of problems $\{\mathbb{P}_{\omega}:\omega\in\Omega
_{n}\}$, Theorem \ref{RESULT: centred, ordered l.s. form}\ showing that every
$\mathbb{P}_{\omega}$\ is affinely equivalent to a centred least-squares problem.

Section \ref{SECTION: inequality constrained varianr} deals with the
inequality constrained case. We say that two minimisation problems with the
same loss function are \textit{effectively equivalent} if they have the same
infimal value attained on the same solution set. A simple continuity argument
shows that, for every $\mathbb{P}_{\omega}$, its weak inequality constraint
variant either has a trivial solution or is effectively equivalent to
$\mathbb{P}_{\omega}$\ itself, Theorem
\ref{RESULT: inequality constrained variant} specifying exactly when each occurs.

Section \ref{SECTION: A p.s.d.}\ discusses the $A\succeq O$ case, Lemma
\ref{RESULT: CLS problem; simplified form}\ showing that every centred
least-squares problem is now affinely equivalent to a simplified form. Theorem
\ref{Theorem 2} establishes that there are at most three possibilities for
such a form, specifying exactly when each occurs. In two of them,
distinguished by whether or not $\widehat{X}$ is empty, $\underline{L}=0$. In
the third, (a) $\underline{L}>0$ while (b) a \textit{reduced form} of
$\mathbb{P}_{\omega}$\ is effectively equivalent to an induced (centred) full
least-squares problem\ in $r<n$\ variables. Remark \ref{REMARK: Newcomb: links with simlts. diagl'n.} establishes direct
points of contact with simultaneous diagonalisation.

Section \ref{SECTION: A p.d.: canonical form}: given the above results, there
is no loss in assuming now that $A$ is positive definite, in which case,
$\mathbb{P}_{\omega}$\ always has a solution (Theorem
\ref{RESULT: A p.d. => there's always a sol'n.}), Theorem
\ref{RESULT: canonical form} establishing that it is affinely -- indeed,
linearly -- equivalent to a \textit{canonical form }$\mathbb{P}_{\omega^{\ast
}}$, governed by the spectral decomposition of $B$.

Section \ref{SECTION: A p.d.: optimised canonical form}: exploiting orthogonal
indeterminacy within eigenspaces, Theorem
\ref{RESULT: canonical form -- optimised}\ establishes that solutions to
$\mathbb{P}_{\omega^{\ast}}$\ can be characterised in terms of those of a
dimension-reduced canonical form $\mathbb{P}_{\omega^{\ast\ast}}$. There is no
loss in restricting attention to \emph{regular} such forms, denoted by $\mathbb{P}%
_{\overline{\omega}^{\ast\ast}}$\ (Remark \ref{REMARK 71}).

Section \ref{SECTION: solving d-r c.f.} presents a series of auxiliary results
culminating in Theorem \ref{RESULT: MAIN final theorem}, specifying the
minimised objective function and solution set for any $\mathbb{P}%
_{\overline{\omega}^{\ast\ast}}$. This completes our primary objective. 

Section \ref{SECTION: Intrinsic instability}\ discusses intrinsic instability of 
solution sets and problem forms, while Section \ref{SECTION: Worked example}\ 
concludes the paper with a short discussion.

\section{Affine equivalence \label{SECTION: Affine equivalence}}

Recalling that the nonsingular affine transformations $G_{n}$ on $R^{n}$\ form
a group under composition, whose general member we denote by $g:x\rightarrow
x_{g}$, the \textit{same} instance of $\mathbb{P}_{\omega}$\ can be
re-expressed in different, affinely equivalent, coordinate systems.

\begin{definition}
For any vector $a$ and for any nonsingular $T$, $g=g_{(T,a)}$\ denotes the map
$x\rightarrow x_{g}:=T^{-1}(x-a)$, inducing $\omega\rightarrow\omega_{g}$ via:%
\[%
\begin{array}
[c]{lllll}%
t & \rightarrow & t_{g} & := & \,T^{-1}(t-a)\\
A & \rightarrow & A_{g} & := & \,T^{\prime}AT,\\
B & \rightarrow & B_{g} & := & \,T^{\prime}BT,\\
b & \rightarrow & b_{g} & := & \,T^{\prime}(b+Ba)\text{ and }\\
k & \rightarrow & k_{g} & := & \,k-a^{\prime}(2b+Ba)\text{.}%
\end{array}
\]
For linear maps ($a=0_{n}$), we abbreviate $g$\ as $g_{T}$.
\end{definition}

\begin{remark} 
Note that: (a) the rank and signature of $A$\ and $B$ are maintained
in those of $A_{g}$\ and $B_{g}$ respectively, (b) $A$\ and $B$ are unchanged
under translation $(T=I)$, and (c) $k$\ is unchanged under linear maps
($a=0_{n}$).
\end{remark} 

\begin{theorem}
\label{Proposition: affine change of coordinates}
For all $\omega\in\Omega_{n}%
$, $x\in R^{n}$ and $g\in G_{n}$: $L_{\omega}(x)=L_{\omega_{g}}(x_{g})$ and
$Q_{\omega}(x)=Q_{\omega_{g}}(x_{g})$, so that $x\in X_{\omega}\Leftrightarrow
x_{g}\in X_{\omega_{g}}$, $\underline{L}_{\omega}=\underline{L}_{\omega_{g}}$,
$\widehat{x}\in\widehat{X}_{\omega}\Leftrightarrow\widehat{x}_{g}%
\in\widehat{X}_{\omega_{g}}$ and $\widehat{X}_{\omega}\neq\emptyset
\Leftrightarrow\widehat{X}_{\omega_{g}}\neq\emptyset$, in which case
$\widehat{L}_{\omega}=\widehat{L}_{\omega_{g}}$.
\end{theorem}

\noindent In view of Theorem \ref{Proposition: affine change of coordinates},
we say that $\mathbb{P}_{\omega}$\ and $\mathbb{P}_{\omega_{g}}$ are
\textit{affinely equivalent}, writing $\omega\sim\omega_{g}$. Again, if
$g=g_{(T,a)}$ with $T$ orthogonal, we call $\mathbb{P}_{\omega}$\ and
$\mathbb{P}_{\omega_{g}}$ \textit{Euclideanly }equivalent.

Recall that $\omega\rightarrow\omega_{-}:=(A,-B,t,-b,-k)$ also leaves
$\mathbb{P}_{\omega}$ unchanged. For later use (Section
\ref{SECTION: A p.d.: canonical form}), we note here

\begin{lemma}
\label{RESULT: constraint sign reversal commutes with a.e.}$\mathbb{P}%
_{\omega}\rightarrow\mathbb{P}_{\omega_{-}}$ commutes with $\mathbb{P}%
_{\omega}\rightarrow\mathbb{P}_{\omega_{g}}$, $g\in G$.
\end{lemma}

\section{Centred least-squares form\label{SECTION: projected least squares re-formulation}}

We characterise here the set of centred least-squares problems to which a
given problem $\mathbb{P}_{\omega}$\ is affinely equivalent. Partitioning $T$
conformably with $B$, as in
(\ref{DISPLAY: (1,0) partitioning of x, B, t & b}), let $\mathcal{T}$\ denote
$\{T:T_{11}$ is orthogonal, $T_{10}=O$ and $T_{00}$\ is nonsingular$\}$,
noting that $\mathcal{T}$\ forms a group under multiplication.

\begin{theorem}
\label{RESULT: centred, ordered l.s. form}
Let $\omega\in\Omega_{n}$ and
$T_{A}:=U_{A}\operatorname*{diag}(D_{A}^{-1},I_{n-r})$ where $A$ has spectral
decomposition $A=U_{A}\operatorname*{diag}(D_{A}^{2},O_{n-r})U_{A}^{\prime}$
with $U_{A}$ orthogonal and $D_{A}$ diagonal, positive definite.
Then: 
\begin{enumerate}
\item[(i)] $\mathbb{P}_{\omega_{g_{0}}}$, $g_{0}=g_{(T_{A},t)}$, is a
centred least-squares problem; \\
\item[(ii)] $\mathbb{P}_{\omega_{g\circ g_{0}}%
}$ is also such a problem if and only if $g=g_{T}$ for some $T\in\mathcal{T}$.
\end{enumerate} 
\end{theorem}

In view of Theorems \ref{Proposition: affine change of coordinates}\ and
\ref{RESULT: centred, ordered l.s. form}, there is no loss in restricting
attention to centred least-squares problems $\mathbb{P}_{\omega}$, in which
case $L_{\omega}(x)=\left\Vert x_{1}\right\Vert ^{2}$.

\section{Inequality contrained variant\label{SECTION: inequality constrained varianr}}

We denote by $\mathbb{P}_{\leq}$ the inequality constrained variant of
$\mathbb{P}$ in which the feasible set required to be nonempty $X:=\{x\in
R^{n}:Q(x)=0\}\neq\emptyset$ is replaced by $X_{\leq}:=\{x\in R^{n}%
:Q(x)\leq0\}\neq\emptyset$, its infimal value and solution set being denoted 
by $\underline{L}_{\leq}$ and $\widehat{X}_{\leq}$ respectively. The reverse
inequality is accommodated by changing the overall sign of $(B,b,k)$.

Affine equivalence generalises at once to the inequality constrained case, as
does Theorem \ref{RESULT: centred, ordered l.s. form}. Accordingly, in
discussing $\mathbb{P}_{\leq}$, there is again no loss in restricting
attention to the centred least-squares case, when $L(x)=\left\Vert
x_{1}\right\Vert ^{2}$.

\begin{theorem}
\label{RESULT: inequality constrained variant}Let $\mathbb{P}_{\leq}$ be an
inequality constrained, centred least-squares problem.
\begin{enumerate}
\item[(i)]
When
$A\succ O$: \\ if $Q(0_{n})\leq0$, $\underline{L}_{\leq}=0$ and
$\widehat{X}_{\leq}=\{0_{n}\}$; \\ otherwise, $X\neq\emptyset
$\ and $\mathbb{P}_{\leq}$ is effectively equivalent to $\mathbb{P}$.
\item[(ii)] When $A\succeq O$, putting $X_{0,\leq}:=\{x_{0}\in
R^{n-r}:Q(0_{r}^{\prime},x_{0}^{\prime})^{\prime}\leq0\}$:\\
if $X_{0,\leq}\neq\emptyset$, $\underline{L}_{\leq}=0$ and $\widehat{X}_{\leq
}=\{(0_{r}^{\prime},x_{0}^{\prime})^{\prime}:x_{0}\in X_{0,\leq}\} $;
\\ otherwise, $X\neq\emptyset$\ and $\mathbb{P}_{\leq}$ is
effectively equivalent to $\mathbb{P}$.
\end{enumerate}
\end{theorem}

\begin{proof}
The proof is similar in both cases.
\begin{enumerate}
\item[(i)]If $Q(0_{n})\leq0$,\ the result
is immediate. Else, $Q(0_{n})>0$, continuity of $Q(\cdot)$\ ensuring that,
$\forall\,x\in X_{\leq}$, $Q(x)<0\Rightarrow\exists\,0<\kappa<1$ with
$Q(\kappa x)=0$.
\item[(ii)]  If $X_{0,\leq}\neq\emptyset$, the result is
immediate. Else, $Q(0_{r}^{\prime},x_{0}^{\prime})^{\prime}>0$ $\forall
\,x_{0}\in R^{n-r}$ while, $\forall\,x\in X_{\leq}$, $Q(x)<0\Rightarrow
\exists\,0<\kappa<1$ with $Q(\kappa x_{1}^{\prime},x_{0}^{\prime})^{\prime}=0$.
\end{enumerate}
\end{proof}

\section{Solving $\mathbb{P}_{\omega}$\ when $A$ is positive semi-definite\label{SECTION: A p.s.d.}}

In this section, we take $A$ positive semi-definite ($r<n$), so that $x_{0}$
occurs in the constraint but not in the objective function. Accordingly, any
centred least-squares problem takes an associated reduced form, an immediate
lemma providing geometric insight.

Denoting orthogonal projection of $R^{n}$\ onto $S_{r}$\ by $P:x\rightarrow
x_{1}$, we have:

\begin{definition}
For any centred least-squares problem $\mathbb{P}_{\omega}$ with $r<n$, its
\textit{reduced form} is:%
\begin{equation}
\text{find }\underline{L}_{1}:=\inf\left\{  L_{1}(x_{1}):x_{1}\in
X_{1}\right\}  ,L_{1}(x_{1}):=\left\Vert x_{1}\right\Vert ^{2},X_{1}%
:=P(X)\text{.}\label{DISPLAY: problem (15)}%
\end{equation}
\end{definition}
\begin{remark} 
Note that $X_{1}$\ is (a) nonempty, since $X$\ is nonempty, and (b) given by%
\[
X_{1}=\left\{  x_{1}\in R^{r}:X_{0}(x_{1})\neq\emptyset\right\}
\]
where $X_{0}(x_{1}):=\{x_{0}\in R^{n-r}:Q(x_{1}^{\prime},x_{0}^{\prime
})^{\prime}=0\}$.
\end{remark}

\begin{lemma}
\label{RESULT: CLS problem: solns. 1-1 with those of its RF}Let $\mathbb{P}%
_{\omega}$ be a centred least-squares problem\ with $r<n$. Then $\forall x\in
R^{n}$, $L(x)=L_{1}(x_{1})$ while $x\in X\Leftrightarrow\lbrack x_{1}\in
X_{1},\;x_{0}\in X_{0}(x_{1})]$, so that $\underline{L}_{1}=\underline{L}$,
while $(x_{1}^{\prime},x_{0}^{\prime})^{\prime}$ solves $\mathbb{P}_{\omega
}\Leftrightarrow\lbrack x_{1}$\ solves (\ref{DISPLAY: problem (15)}) and
$x_{0}\in X_{0}(x_{1})]$.
\end{lemma}
\noindent Geometrically, the reduced form seeks the infimal (squared) distance
from the origin in $S_{r}$ to the orthogonal projection of the conic $X$ onto
that subspace, its solution set $\widehat{X}_{1}$\ being the orthogonal
projection of $\widehat{X}$ onto $S_{r}$.

To help solve the reduced form (\ref{DISPLAY: problem (15)}), we introduce a
simplifying linear transformation, via a decomposition of the null space of
$A$ according to its intersections with the range and null spaces of $B_{00}$.

\begin{definition}\label{DEF52}
Let $\mathbb{P}_{\omega}$ be a centred least-squares problem\ with $r<n$.
Then, $\mathbb{P}_{\omega}$ is said to be \textit{in simplified form} if, for
some $0\leq s_{0}\leq n-r$ and for some nonsingular diagonal $\Gamma_{0} $ of
order $s_{0}$, $B$ has the partitioned form:%
\begin{equation}
B=\begin{pmatrix}%
B_{11} & C_{10} & O\\
C_{10}^{\prime} & O_{n-r-s_{0}} & O\\
O & O & \Gamma_{0}%
\end{pmatrix}.\label{DISPLAY: B in the simplified form problem}%
\end{equation}
Accordingly, any term involving $\Gamma_{0}$\ is absent if and only
if $B_{00}=O$, and any involving $C_{10}$ if and only if $B_{00}$\ is
nonsingular -- in particular, if $B$\ is (positive or negative) definite.
\end{definition}

\begin{lemma}
\label{RESULT: CLS problem; simplified form}Let $\mathbb{P}_{\omega}$ be a
centred least-squares problem\ with $r<n$. Then, $\exists\;T\in \mathcal{T}$ with
$\mathbb{P}_{\omega_{g_{T}}}$ in simplified form.
\end{lemma}

\begin{proof}
Let $B_{00}$\ have rank $0\leq s_{0}\leq n-r$ and spectral decomposition
$U_{0}\operatorname*{diag}(O_{n-r-s_{0}},\Gamma_{0})U_{0}^{\prime}$, with
$U_{0}$ orthogonal and $\Gamma_{0}$ nonsingular diagonal. Define $C_{10}$ and
$D_{10}$ implicitly via%
\[
\operatorname*{diag}(I_{r},U_{0})^{\prime}B\operatorname*{diag}(I_{r}%
,U_{0})=\begin{pmatrix}
B_{11} & C_{10} & D_{10}\\
C_{10}^{\prime} & O_{n-r-s_{0}} & O\\
D_{10}^{\prime} & O & \Gamma_{0}%
\end{pmatrix}
\]
and put%
\begin{equation}
T:=\operatorname*{diag}(I_{r},U_{0})
\begin{pmatrix}
I_{r} & O & O\\
O & I_{n-r-s_{0}} & O\\
-\Gamma_{0}^{-1}D_{10}^{\prime} & O & I_{s_{0}}%
\end{pmatrix}. \label{eqnew2015}%
\end{equation}
Then $T\in\mathcal{T}$ and so, by Theorem \ref{RESULT: centred, ordered l.s. form}, $\mathbb{P}%
_{\omega_{g_{T}}}$\ is a centred least-squares problem in simplified form.
\end{proof}

We note in passing that, while preserving simplified form, a further linear transformation establishes direct points of contact with simultaneous diagonalisation.

\begin{remark}
\label{REMARK: Newcomb: links with simlts. diagl'n.} If $\mathbb{P}_\omega$ is in simplified form, while $B_{11}$ in
(\ref{DISPLAY: B in the simplified form problem}) has spectral decomposition:
\[
B_{11}=U_{1}\operatorname*{diag}(O_{r-s_{1}},\Gamma_{1})U_{1}^{\prime},\;\;\Gamma_{1}\text{\ nonsingular,}%
\]
the further transformation $x\rightarrow T^{-1}x$ with $T:=\operatorname*{diag}%
(U_{1},I_{n-r}) \in \mathcal{T}$ induces:%
\begin{equation}
B\rightarrow
\begin{pmatrix}
O_{r-s_{1}} & O & E_{10} & O\\
O & \Gamma_{1} & F_{10} & O\\
E_{10}^{\prime} & F_{10}^{\prime} & O_{n-r-s_{0}} & O\\
O & O & O & \Gamma_{0}%
\end{pmatrix}
\text{ in which }\begin{pmatrix} 
E_{10}\\
F_{10}%
\end{pmatrix}
=U_{1}^{\prime}C_{10}%
\label{DISPLAY: B a la Newcomb in the simplified form problem}%
\end{equation}
so that, using again Theorem \ref{RESULT: centred, ordered l.s. form}, $\mathbb{P}_{\omega_g}$, $g=g_T$, remains in its simplified form. 
This can be seen as extending Newcomb \cite{Newcomb} who showed that any two symmetric
matrices, neither of which is indefinite, can be simultaneously diagonalised.
For, there is no loss in restricting attention to matrices that, like $A$, are
either positive definite or positive semi-definite, which, if true of $B$,
entails that $E_{10}$ and $F_{10}$ in
(\ref{DISPLAY: B a la Newcomb in the simplified form problem}) are absent or
zero respectively.
\end{remark}

Returning to the mainstream, the following additional terms are used.

\begin{definition}\label{DEF53}
For any centred least-squares problem $\mathbb{P}_{\omega}$ in simplified
form, we sub-partition $x_{0}$ and $b_{0}$ so that
\[
x_{0}=
\begin{pmatrix}
y_{0}\\
z_{0}%
\end{pmatrix}
,\;b_{0}=
\begin{pmatrix}
c_{0}\\
d_{0}%
\end{pmatrix}
\text{ \ and \ }B_{00}=
\begin{pmatrix}
O_{n-r-s_{0}} & O\\
O & \Gamma_{0}%
\end{pmatrix}
\]
conform. Accordingly, any term involving $y_{0}$\ or $c_{0}$\ is absent if and
only if $B_{00}$ is nonsingular; and any involving $z_{0}$\ or $d_{0}$\ if and
only if $B_{00}=O$. 

For given $x_{1}$, $Q(\cdot)$\ depends quadratically on
$z_{0}$, but only linearly on $y_{0}$, since%
\begin{equation}
Q(x)\equiv(z_{0}+\Gamma_{0}^{-1}d_{0})^{\prime}\Gamma_{0}(z_{0}+\Gamma
_{0}^{-1}d_{0})+2(C_{10}^{\prime}x_{1}+c_{0})^{\prime}y_{0}+Q_{1}%
(x_{1})\label{DISPLAY: CLS  in simplified form: Q identity}%
\end{equation}
in which $Q_{1}(x_{1}):=x_{1}^{\prime}B_{11}x_{1}+2b_{1}^{\prime}x_{1}-k_{1},$
\ with \ $k_{1}=k_{1}(\omega):=k+d_{0}^{\prime}\Gamma_{0}^{-1}d_{0}$.

\noindent If $B_{00}\neq O$, $Z_{0}(\alpha):=\{z_{0}\in R^{s_{0}%
}:(z_{0}+\Gamma_{0}^{-1}d_{0})^{\prime}\Gamma_{0}(z_{0}+\Gamma_{0}^{-1}%
d_{0})=\alpha\}$, $\alpha\in R$, so that $\alpha(\Gamma_{0}):=\{\alpha
:Z_{0}(\alpha)\neq\emptyset\}$\ is $R$, $[0,\infty)$\ or $(-\infty
,0]$\ according as $\Gamma_{0}$\ is indefinite, positive definite or negative
definite, while $Z_{0}(0)=\{-\Gamma_{0}^{-1}d_{0}\}$\ if $\Gamma_{0}$\ is
definite. If $B_{00}$ is singular and $c_{0}\neq0$, $\alpha(y_{0}%
):=k_{1}-2c_{0}^{\prime}y_{0}$, while $y_{0}(c_{0}):=k_{1}c_{0}/(2\left\Vert
c_{0}\right\Vert ^{2})$, so that $\alpha(y_{0}(c_{0}))=0$.
\end{definition}

In view of previous results, when $r<n$, it suffices to solve $\mathbb{P}%
_{\omega}$\ -- or, giving also each $X_{0}(x_{1})$ $(x_{1}\in X_{1})$, its
reduced form (\ref{DISPLAY: problem (15)}) -- for centred least-squares
problems in simplified form.

\begin{definition}
Let $\mathbb{P}_{\omega}$ with $r<n$ be a centred least-squares problem in
simplified form. Then, $\omega_{1}:=(I_{r},B_{11},0_{r},b_{1},k_{1})$ is
called the \textit{projected form} of $\omega=(\operatorname*{diag}%
(I_{r},O),B,0_{n},b,k)$, and we say that $\mathbb{P}_{\omega}$ admits:
\begin{enumerate}
\item[(a)] \textit{a perfect solution} if $L(x)=0$\ for some $x\in X$;
\item[(b)] \textit{an essentially perfect solution}\ if $L(x)>0$ $(x\in X)$, but
$\underline{L}=0$;
\item[(c)] \textit{a projected, yet imperfect, reduced
form}\ if (i) $X_{\omega_{1}}\neq\emptyset$, so that $\omega_{1}\in\Omega_{r}%
$, and (ii) its reduced form is effectively equivalent to $\mathbb{P}%
_{\omega_{1}}$:
\[
\text{find }\inf\left\{  L_{1}(x_{1}):x_{1}\in X_{\omega_{1}}\right\}
,\;X_{\omega_{1}}=\left\{  x_{1}\in R^{r}:Q_{1}(x_{1})=0\right\}  \text{,}%
\]
in which $k_{1}\neq0$.
\end{enumerate} 
\end{definition}

\begin{example} \label{EXAMPLE 51}
Examples of these three possibilities are the problems, with $r=1$ and $n=2$,
of finding the infimum of $x_{1}^{2}$\ over all $(x_{1},x_{0})^{\prime}%
$\ satisfying, respectively, (a) $x_{1}^{2}-x_{0}^{2}=-1$, (b) $x_{1}x_{0}=1$
and (c) $x_{1}^{2}-x_{0}^{2}=+1$, the corresponding sets $X_{1}$\ being $R$,
$R\backslash\{0\}$\ and $\{x_{1}:x_{1}^{2}\geq1\}$.
\end{example}
\noindent Indeed, there are no other possibilities. Denoting the closure and
boundary of $X_{1}$\ by $cl(X_{1})$ and $\partial X_{1}$, we have:

\begin{theorem}
\label{Theorem 2}Let $\mathbb{P}_{\omega}$ with $r<n$ be a centred
least-squares problem in simplified form. There are at most three, mutually
exclusive, possibilities:
\begin{enumerate} 
\item[(a)] $\mathbb{P}_{\omega}$ admits a perfect solution,
\item[(b)] $\mathbb{P}_{\omega}$ admits an essentially perfect solution,
\item[(c)] $\mathbb{P}_{\omega}$ admits a projected, yet imperfect, reduced form,
\end{enumerate}
whose occurrences are characterised thus:
\begin{enumerate}
\item[(a)] occurs  $\Leftrightarrow0_{r}\in X_{1}$ $\Leftrightarrow\lbrack\underline{L}=0,\widehat{X}\neq\emptyset]$,
\item[(b)] occurs  $\Leftrightarrow0_{r}\in\partial X_{1}$ $\Leftrightarrow\lbrack\underline{L}=0,\widehat{X}=\emptyset]$,
\item[(c)] occurs  $\Leftrightarrow0_{r}\notin cl(X_{1})$ $\Leftrightarrow\underline{L}>0$,
\end{enumerate}
these arising as follows:
\begin{itemize}
\item[(i)] If $s_{0}=0$:
\[
\text{(a) occurs}\Leftrightarrow\lbrack k_{1}=0,\;c_{0}=0]\text{\ or }%
c_{0}\neq0\text{,}%
\]
$X_{0}(0)$ being $R^{n-r}$ or $\{y_{0}(c_{0})\}$ respectively. Otherwise,
$[k_{1}\neq0,\;c_{0}=0]$, with (c) or (b) occurring according as $C_{10}$
does, or does not, vanish. When (c) occurs, $X_{1}=X_{\omega_{1}}$ with each
$X_{0}(x_{1})=R^{n-r}$.\bigskip

\item[(ii)] If $s_{0}=n-r$:
\[
\text{(a) occurs}\Leftrightarrow k_{1}=0\text{\ or }k_{1}\Gamma_{0}\text{ has
a positive eigenvalue,}%
\]
$X_{0}(0)$ being $Z_{0}(0)$ or $Z_{0}(k_{1})$ respectively, while (b) does
\textbf{not} occur. Thus, (c) occurs $\Leftrightarrow(-k_{1}\Gamma
_{0})\succ O$, when $X_{1}=\{x_{1}\in R^{r}:k_{1}Q_{1}(x_{1})\geq0\}$ with
each $X_{0}(x_{1})=\{x_{0}:z_{0}\in Z_{0}(-Q_{1}(x_{1}))\}$.\bigskip

\item[(iii)] If $0<s_{0}<n-r$:
\[%
\begin{tabular}[c]{ll}%
$\text{(a) occurs}\Leftrightarrow$ & $[k_{1}=0,\;c_{0}=0]$ or
$[k_{1}\Gamma_{0}\text{ has a positive eigenvalue, }c_{0}=0]\text{ }$\\
& $\text{or }c_{0}\neq0\text{,}$%
\end{tabular}
\]
$X_{0}(0)$ comprising all $x_{0}$\ with $z_{0}$ in $Z_{0}(0)$, $Z_{0}(k_{1})$
or $\cup_{\alpha(y_{0})\in\alpha(\Gamma_{0})}Z_{0}(\alpha(y_{0}))$,
respectively. Otherwise, $[(-k_{1}\Gamma_{0})\succ O,\;c_{0}=0]$, with (c) or
(b) occurring according as $C_{10}$ does, or does not, vanish. When (c) occurs,
$X_{1}$ and each $X_{0}(x_{1})$ are as in (ii).

\end{itemize}
\end{theorem}

\begin{proof}
Using Lemma \ref{RESULT: CLS problem: solns. 1-1 with those of its RF}, the
characterisations of the occurrence of (a) and of (b) are immediate. Since
\begin{equation}
0_{r}\notin cl(X_{1})\Leftrightarrow\lbrack\exists\;\eta>0\text{ such
that\ }\left\Vert x_{1}\right\Vert <\eta\Rightarrow x_{1}\notin X_{1}%
]\label{DISPLAY: 0.notin.closure(X_1): alternative form}%
\end{equation}
so, too, is the fact that (c) occurs $\Rightarrow0_{r}\notin cl(X_{1}%
)\Leftrightarrow\underline{L}>0$. Recall that%
\begin{equation}
x_{1}\notin X_{1}\Leftrightarrow Q(x_{1}^{\prime},x_{0}^{\prime})^{\prime}%
\neq0\text{ for all }x_{0}\in R^{n-r}\text{.}%
\label{DISPLAY: x_1 not in X_1 iff Q(x_1, x_0) never 0}%
\end{equation}
Using the $x_{1}=0$ instances of
(\ref{DISPLAY: CLS in simplified form: Q identity}) and
(\ref{DISPLAY: x_1 not in X_1 iff Q(x_1, x_0) never 0}), and identifying the
former as an instance of the identity
(\ref{DISPLAY: Q: Lemma 1.1's basic indentity}), Lemma
\ref{RESULT: when does constraint NOT have a solution?} gives:%
\[
0_{r}\notin X_{1}\Leftrightarrow\left\{
\begin{tabular}[c]{llll}%
\ \ \ \ \ (i): & $[s_{0}=0,$ & $c_{0}=0,$ & $k_{1}\neq0],$\\
\thinspace or\ (ii): & $[s_{0}=n-r,$ & $(-k_{1}\Gamma_{0})\succ O],$ & \\
or (iii): & $[0<s_{0}<n-r,$ & $(-k_{1}\Gamma_{0})\succ O,$ & $c_{0}=0]$%
\end{tabular}
\right\}  .
\]
Similarly, using also (\ref{DISPLAY: 0.notin.closure(X_1): alternative form}),
we have:%
\[
0_{r}\notin cl(X_{1})\Leftrightarrow\lbrack0_{r}\notin X_{1}\text{ and }%
C_{10}\text{\ is either absent }(s_{0}=n-r)\text{ or zero}]\text{.}%
\]
Suppose $0_{r}\notin cl(X_{1})$. If also $s_{0}=0$, $Q(x)=Q_{1}(x_{1})$, so
that $X_{\omega_{1}}=X_{1}\neq\emptyset$ and the reduced form
(\ref{DISPLAY: problem (15)}) of $\mathbb{P}_{\omega}$\ \textit{is}
$\mathbb{P}_{\omega_{1}}$. Thus, (c) occurs, while each $X_{0}(x_{1})=R^{n-r}%
$. If, instead, we also have $s_{0}>0$, $(-k_{1}\Gamma_{0})\succ O$ and
\[
k_{1}Q(x)=k_{1}Q_{1}(x_{1})-(z_{0}+\Gamma_{0}^{-1}d_{0})^{\prime}(-k_{1}%
\Gamma_{0})(z_{0}+\Gamma_{0}^{-1}d_{0}),
\]
so that $X_{1}=\{x_{1}:k_{1}Q_{1}(x_{1})\geq0\}$. Since $X_{1}\neq\emptyset$
while $k_{1}\neq0$, Theorem \ref{RESULT: inequality constrained variant}%
\ establishes that (c) again occurs, while each $X_{0}(x_{1})$\ has the form
stated. The forms taken by $X_{0}(0)$\ when (a) occurs are immediate.
\end{proof}

\begin{corollary}
Under the hypotheses of Theorem \ref{Theorem 2},\ and adopting its
terminology, it is necessary for (b) to occur that $B$ be indefinite.
\end{corollary}

\begin{proof}
If $B$ is not indefinite, $C_{10}$\ is either absent or zero.
\end{proof}
\noindent Geometrically, a hyperbolic quadratic constraint is necessary for
$\mathbb{P}$\ to admit an essentially perfect solution.

As Theorem \ref{Theorem 2} shows -- and as the example of (b) given
immediately before it illustrates -- it is possible that $\mathbb{P}_{\omega}$
has no solution when $r<n$, reflecting the fact that, although $X$\ is closed,
its projection $X_{1}=P(X)$\ may not be. In contrast, a simple compactness
argument, used in proving Theorem
\ref{RESULT: A p.d. => there's always a sol'n.} below, shows that
$\mathbb{P}_{\omega}$ \textit{always} has a solution when $r=n$.

\section{Canonical form \label{SECTION: A p.d.: canonical form}}

Given the above results, it suffices to solve the centred, full least-squares
form $A=I$, $t=0$ of problem $\mathbb{P}_{\omega}$. The affinely constrained
case being well-known, there is no essential loss in also requiring $B\neq O$.
Reversing if necessary the overall sign of $(B,b,k) $, and recalling Lemma
\ref{RESULT: constraint sign reversal commutes with a.e.}, there is no loss in
further assuming that $B$ has a positive eigenvalue.

Such a problem\ always has a solution.

\begin{theorem}
\label{RESULT: A p.d. => there's always a sol'n.}$\widehat{X}_{\omega}%
\neq\emptyset$ for any full least-squares problem $\mathbb{P}_{\omega}$.
\end{theorem}

\begin{proof}
For any $\widetilde{x}\in X_{\omega}$, $\mathbb{P}_{\omega}$\ is essentially
equivalent to minimising $\left\Vert x-t\right\Vert ^{2}$\ over the compact
set $X_{\omega} \cap \{  x\in R^{n}:\left\Vert x-t\right\Vert
\leq\left\Vert \widetilde{x}-t\right\Vert \}$.
\end{proof}
\noindent Geometrically, there is always a shortest normal from $t$ to the
conic $X$.

We introduce next the notion of a canonical form, a full least-squares form of
$\mathbb{P}_{\omega}$\ in which the constraint has been simplified according
to the spectral decomposition of $B$. In it, $u_{m}:=(1,0_{m-1}^{\prime
})^{\prime}$ denotes the first unit coordinate vector in $R^{m}$.

\begin{definition}
\label{DEFINITION: canonical form}Let $\mathbb{P}_{\omega}$\ be a centred,
full least-squares problem in which $B$ has a positive eigenvalue. Let $B$
have rank $s$ and distinct nonzero eigenvalues $\gamma_{1}>...>\gamma_{q}$
$(1\leq q\leq s)$, so that $\gamma_{1}>0$. For each $i=1,...,q$, let
$\gamma_{i}$ have multiplicity $m_{i}\geq1$ and let the orthogonal projection
of $b$\ onto the corresponding eigenspace have length $l_{i}\geq0 $. Let
$m_{0}:=n-s\geq0$ denote the dimension of the null space of $B$, and
$l_{0}\geq0$ the length of the orthogonal projection of $b$\ onto this
subspace. \smallskip\newline
With $\varepsilon:=l_{0}$\ and $\delta=(\delta_i) \in R^q$, $\delta
_{i}:=\left\vert \gamma_{i}\right\vert ^{-1}l_{i}$, let
$\omega^{\ast}:=(A^{\ast},B^{\ast},t^{\ast},b^{\ast},k^{\ast})$ in which
$A^{\ast}=I $, $B^{\ast}=\operatorname*{diag}(O_{m_{0}},\Gamma)$ where
$\Gamma:=\operatorname*{diag}(\gamma_{1}I_{m_{1}},...,\gamma_{q}I_{m_{q}})$,
\[
t^{\ast}=
\begin{pmatrix}
0_{m_{0}}\\
\delta_{1}u_{m_{1}}\\
\vdots\\
\delta_{q}u_{m_{q}}%
\end{pmatrix}
, b^{\ast}=
\begin{pmatrix}
\varepsilon u_{m_{0}}\\
0_{m_{1}}\\
\vdots\\
0_{m_{q}}%
\end{pmatrix}
\]
and $k^{\ast}=k+%
{\textstyle\sum\nolimits_{i=1}^{q}}
\gamma_{i}^{-1}l_{i}^{2}$. Then, we call $\mathbb{P}_{\omega^{\ast}}$ the
\textit{canonical form} of $\mathbb{P}_{\omega}$.
\end{definition}
\noindent The reason for this terminology is made clear by the following result.

\begin{theorem}
\label{RESULT: canonical form}$\mathbb{P}_{\omega}$\ is Euclideanly equivalent
to $\mathbb{P}_{\omega^{\ast}}$.
\end{theorem}

\begin{proof}
By the spectral decomposition of $B$, there is an orthogonal matrix $T_{B}$
with $T_{B}^{\prime}BT_{B}=\operatorname*{diag}(O_{m_{0}},\Gamma)$, where
$T_{B}$\ is unique up to postmultiplication by $U\in\mathcal{U}%
:=\{\operatorname*{diag}(U_{0},U_{1},...,U_{q}):$ $U_{i}^{\prime}%
U_{i}=I_{m_{i}}\;(0\leq i\leq q)\}$.

Let $T_{B}^{\prime}b=(c^{\prime},d_{1}^{\prime},...,d_{q}^{\prime})^{\prime}$
and $T_{B}^{\prime}BT_{B}$ conform, so that $\left\Vert c\right\Vert =l_{0}%
$\ while $\left\Vert d_{i}\right\Vert =l_{i}$ $(1\leq i\leq q)$, and let now
$U=U_{B}\in\mathcal{U}$ be such that the first column of $U_{0}$ is
$c/\left\Vert c\right\Vert $ if $c\neq0$ and of $U_{i}$ $(1\leq i\leq q)$ is
$d_{i}/\left\Vert d_{i}\right\Vert $ if $d_{i}\neq0$.

Finally, let $g:x\rightarrow x_{g}=U_{B}^{\prime}\{T_{B}^{\prime}x+(0^{\prime
},\gamma_{1}^{-1}d_{1}^{\prime},...,\gamma_{q}^{-1}d_{q}^{\prime})^{\prime}%
\}$. Then, using Lemma \ref{RESULT: when does constraint NOT have a solution?}%
, $\mathbb{P}_{\omega}$\ is Euclideanly equivalent to $\mathbb{P}_{\omega_{g}%
}=\mathbb{P}_{\omega^{\ast}}$.
\end{proof}

\begin{remark}
Viewed geometrically, the isometry inducing $\omega\rightarrow\omega^{\ast}$
so translates, rotates and reflects coordinate axes that: \smallskip
\newline(a) $B^{\ast}$, determining the quadratic part of the constraint,
becomes diagonal,\smallskip\newline(b) $b^{\ast}$, determining its linear
part, vanishes when $B^{\ast}$\ is nonsingular and, otherwise, has at most one
nonzero element -- its coordinate $\varepsilon\geq0$ on the first dimension of
the null space of $B^{\ast}$ -- and: \smallskip\newline(c) the target
$t^{\ast}$\ is \textit{orthogonal} to this subspace, and has at most one
nonzero element associated with each nonzero eigenvalue of $B^{\ast}$ -- its
coordinate $\delta_{i}\geq0$ on the first dimension of the corresponding eigenspace.
\end{remark}

As\ is explicit in the proof of Theorem \ref{RESULT: canonical form}, the
simple structure enjoyed by a canonical form within each eigenspace of
$B^{\ast}$ is made possible by a well-known orthogonal indeterminacy there.
This simple structure invites a natural dimension reduction, further
exploiting this same indeterminacy.

\section{Dimension-reduced canonical form\label{SECTION: A p.d.: optimised canonical form}}

Solutions to $\mathbb{P}_{\omega^{\ast}}$ can be characterised in terms of
those of a dimension-reduced canonical form $\mathbb{P}_{\omega^{\ast\ast}}$,
having just one variable for each distinct eigenvalue of $B^{\ast}$. Variables
$z=(z_{i})\in R^{q}$\ enter the constraint \textit{quadratically}.\ When
$B^{\ast}$\ is singular, an additional scalar variable $y$ associated with its
null space enters the constraint \textit{linearly}.

A unified account is provided in which we denote these variables by $w\in
R^{\underline{n}}$, whether $B^{\ast}$\ is nonsingular $(m_{0}=0)$ or not
$(m_{0}>0)$. In the former case, $\underline{n}=q$\ and $w=z$. In the latter,
$\underline{n}=q+1$ and $w=(y,z^{\prime})^{\prime}$.

\begin{definition}
\label{DEFINITION: dimn-reduced canonical form}Adopting the assumptions and
notation of Definition \ref{DEFINITION: canonical form}, let $\Gamma^{\ast}:=\operatorname*{diag}%
(\gamma_{1},...,\gamma_{q})$, so that $\Gamma^{\ast}$ is either positive
definite $(\gamma_{q}>0)$ or nonsingular indefinite $(\gamma_{1}>0>\gamma_{q}%
$, requiring $q>1)$. \smallskip\newline If $B^{\ast}$ is singular, let
$w_{0}:=(0,\delta^{\prime})^{\prime}$\ and $\Delta:=\operatorname*{diag}%
(0,\Gamma^{\ast})$ conform with $w:=(y,z^{\prime})^{\prime}$, while
$d:=\varepsilon u_{q+1}$. Otherwise, put $w:=z$, $w_{0}:=\delta$%
,\ $\Delta:=\Gamma^{\ast}$ and $d:=0_{q}$. Then, the \textit{dimension-reduced
canonical form} of $\mathbb{P}_{\omega^{\ast}}$\ is $\mathbb{P}_{\omega
^{\ast\ast}}$, $\omega^{\ast\ast}:=(A^{\ast\ast},B^{\ast\ast},t^{\ast\ast
},b^{\ast\ast},k^{\ast\ast})$, with $A^{\ast\ast}=I_{\underline{n}}$,
$B^{\ast\ast}=\Delta$, $t^{\ast\ast}=w_{0}$, $b^{\ast\ast}=d$\ and
$k^{\ast\ast}=k^{\ast}$. That is, the problem $\mathbb{P}_{\omega^{\ast\ast}}%
$\ is to:
\begin{align}
\text{find }\underline{L}^{\ast}  & :=\inf_{w\in R^{\underline{n}}}%
L^{\ast}(w),\;L^{\ast}(w):=\left\Vert w-w_{0}\right\Vert ^{2},\nonumber\\
\text{and }\widehat{W}  & :=\{\widehat{w}\in W:L^{\ast}(w)=\underline{L}%
^{\ast}\}\nonumber\\
\text{subject to }Q^{\ast}(w)  & :=w^{\prime}\Delta w+2d^{\prime}%
w-k^{\ast}=0\text{,}\label{DISPLAY: dimn-reduced canonical form}%
\end{align}
where the feasible and solution sets $W:=\{w\in R^{\underline{n}}:Q^{\ast
}(w)=0\}$\ and $\widehat{W}$\ are nonempty, as $X_{\omega}$ and $\widehat{X}%
_{\omega}$\ are so. We may write $\widehat{w}\in\widehat{W}$\ as
$\widehat{z}=(\widehat{z}_{i})$ $(m_{0}=0)$, or $(\widehat{y},\widehat{z}%
^{\prime})^{\prime}$ $(m_{0}>0)$. Reflecting the form of the conic, when
$d=0$, we call the constraint in $\mathbb{P}_{\omega^{\ast\ast}}$
\textit{elliptic} or \textit{hyperbolic} according as $\gamma_{q}>0$ or
$\gamma_{q}<0$. When $d\neq0$, we call it \textit{parabolic-elliptic} or
\textit{parabolic-hyperbolic} in the same two cases.
\end{definition}
\begin{remark}
When each eigenvalue of $B$ -- equivalently, of $B^{\ast
}$ -- is simple, no such dimension reduction is possible, and $\mathbb{P}%
_{\omega^{\ast\ast}}$\ coincides with $\mathbb{P}_{\omega^{\ast}}$. That is,%
\[
\lbrack s=q\text{ and }m_{0}\in\{0,1\}]\Rightarrow\lbrack\underline{n}=n\text{
and }\omega^{\ast\ast}=\omega^{\ast}].
\]
\end{remark}
\begin{theorem}
\label{RESULT: canonical form -- optimised}Let $\widehat{x}:=(\widehat{x}%
_{0}^{\prime},\widehat{x}_{1}^{\prime},...,\widehat{x}_{q}^{\prime})^{\prime}%
$, where $\widehat{x}_{i}\in R^{m_{i}}$ $(0\leq i\leq q)$.\medskip\ Then,
$\widehat{x}\ $solves $\mathbb{P}_{\omega^{\ast}}$ if and only if%
\[
\widehat{x}_{0}=\widehat{y}u_{m_{0}}\text{ and, for }i=1,...,q\text{,
}\left\{
\begin{tabular}
[c]{ll}%
$\widehat{x}_{i}=\widehat{z}_{i}u_{m_{i}}$ & if $\delta_{i}>0$\\
$\left\Vert \widehat{x}_{i}\right\Vert ^{2}=\widehat{z}_{i}^{2}$ & if
$\delta_{i}=0$%
\end{tabular}
\right\}  ,
\]
where $\widehat{w}=(\widehat{y},\widehat{z}^{\prime})^{\prime}$ solves
$\mathbb{P}_{\omega^{\ast\ast}}$. In particular, $\underline{L}^{\ast
}=\underline{L}$.
\end{theorem}
\noindent Clearly, $\delta_{i}>0\Rightarrow\widehat{z}_{i}\geq0$. Moreover,
Theorem \ref{RESULT: canonical form -- optimised}\ has the immediate

\begin{corollary}
\label{RESULT: charc'n of when d-r c.f. solution determines multiple c.f. solns}%
$\widehat{w}$\ solving $\mathbb{P}_{\omega^{\ast\ast}}$\ determines more than
one solution to $\mathbb{P}_{\omega^{\ast}}\Leftrightarrow\lbrack\delta_{i}=0$
and $\widehat{z}_{i}\neq0]$ for some $i$, in which case the sign of
$\widehat{z}_{i}$\ is indeterminate, while $\widehat{x}_{i}$\ is orthogonally indeterminate.
\end{corollary}

\begin{example}
\label{EXAMPLE: min dist centre to sphere}A clear example of Corollary
\ref{RESULT: charc'n of when d-r c.f. solution determines multiple c.f. solns}%
\ is when the problem $\mathbb{P}_{\omega^{\ast}}$\ is to minimise the
(squared) distance to the unit sphere in $R^{n}$\ from its centre, the origin
$0_{n}$. The solution set $\widehat{X}_{\omega^{\ast}}$\ is, of course, the
unit sphere, while $\mathbb{P}_{\omega^{\ast\ast}}$\ is one-dimensional with
$\widehat{X}_{\omega^{\ast\ast}}=\{\pm1\}$.
\end{example}

\begin{remark}\label{REMARK 71}
If $B^{\ast}$\ is singular, but $b^{\ast}$ has zero component in its null
space -- that is, if $m_{0}>0$,\ but $\varepsilon=0$ -- the linear term in the
constraint vanishes so that the optimal $\widehat{y}=0$, reducing
$\mathbb{P}_{\omega^{\ast\ast}}$ to the corresponding $m_{0}=0$\ case. It
suffices, then, to solve dimension-reduced canonical forms which are \textit{regular}
in the sense defined next.
\end{remark}

\begin{definition}\label{DEFINITION 72} 
A dimension-reduced canonical form is called \textit{regular} if either
$m_{0}=0$ or $[m_{0}>0$ and $\varepsilon>0]$. We denote such forms by
$\mathbb{P}_{\overline{\omega}^{\ast\ast}}$.
\end{definition}
\noindent Since, trivially, $m_{0}=0\Rightarrow\varepsilon=0$, we have at once

\begin{remark}
For any regular dimension-reduced canonical form $\mathbb{P}_{\overline
{\omega}^{\ast\ast}}$:%
\[
\varepsilon>0\Leftrightarrow m_{0}>0\Leftrightarrow d\neq0\Leftrightarrow
\text{ the linear term in }Q^{\ast}(\cdot)\text{\ does NOT vanish,}%
\]
in which case, substituting out the constraint, $\mathbb{P}_{\overline{\omega
}^{\ast\ast}}$ can be rephrased as the unrestricted minimisation of a quartic
in $z$.
\end{remark}

\section{Solving $\mathbb{P}_{\overline{\omega}^{\ast\ast}}$\label{SECTION: solving d-r c.f.}}

We solve any regular dimension-reduced canonical form\ via a series of simple,
insightful, auxiliary results. The first shows that a Lagrangian approach
establishes sufficient conditions for this.

\subsection{Sufficient conditions}

The Lagrangian $\mathcal{L}:=L^{\ast}(w)-\lambda Q^{\ast}(w)$\ for
$\mathbb{P}_{\overline{\omega}^{\ast\ast}}$\ has normal equations%
\begin{equation}
\lbrack I-\lambda\Delta]w=(w_{0}+\lambda d)\label{DISPLAY: normal equations}%
\end{equation}
and Hessian $H:=2[I-\lambda\Delta]$.

\begin{definition}
We refer to $\Lambda:=\{\lambda:H$ has no negative eigenvalues$\}$ as the
\textit{admissible region} -- that is:
\begin{equation}
\Lambda:=(-\infty,\gamma_{1}^{-1}]\text{ if }\gamma_{q}>0\text{,\ while
}\Lambda:=[\gamma_{q}^{-1},\gamma_{1}^{-1}]\text{\ if }\gamma_{q}%
<0\label{DISPLAY: form of LAMBDA}%
\end{equation}
-- its interior, denoted by $\Lambda^{\circ}$, being where $H\succ O$. For each
$\lambda\in\Lambda$,%
\[
W_{N}(\lambda):=\{w\in R^{\underline{n}}:(\ref{DISPLAY: normal equations}%
)\text{ holds}\}\text{,}%
\]
so that $W(\lambda):=W\cap W_{N}(\lambda)$ denotes the set of feasible
vectors, if any, obeying the normal equations.
\end{definition}

\begin{lemma}
\label{RESULT: P_omega**: suff't. conditions to solve}For any $\lambda
\in\Lambda$, $W(\lambda)\neq\emptyset\Rightarrow W(\lambda)=\widehat{W}$.
Indeed, if $w_{N}\in W(\lambda)$,
\[
\underline{L}^{\ast}=\left\Vert w_{N}-w_{0}\right\Vert ^{2}\text{ and
\ }\widehat{W}=\{w\in W:H(w-w_{N})=0\}=W(\lambda).
\]
In particular, if $\lambda\in\Lambda^{\circ}$, $w_{N}$\ \textbf{uniquely}
solves $\mathbb{P}_{\overline{\omega}^{\ast\ast}}$.
\end{lemma}

\begin{proof}
It suffices to note that, by the Cosine Law:%
\[
\forall w\in W,\;L^{\ast}(w)-L^{\ast}(w_{N})=(w-w_{N})^{\prime}[I-\lambda
\Delta](w-w_{N})\geq0\text{.}%
\]

\end{proof}

In view of Lemma \ref{RESULT: P_omega**: suff't. conditions to solve},
$\mathbb{P}_{\overline{\omega}^{\ast\ast}}$\ is completely resolved if we can
find, in explicit form, a nonempty set $W(\lambda)$ for any $\lambda\in
\Lambda$.

\subsection{Feasible solutions to the normal equations\label{SECTION: Feasible solutions to the normal equations}}

Characterising when feasible solutions to the normal equations exist
requires\ extensions of terminology established in Section
\ref{SECTION: A p.d.: optimised canonical form}. Continuing its unified
presentation of $\underline{n}=q$\ and $\underline{n}=q+1$, terms involving
any of $y$, $\widehat{y}$\ or, introduced below, $\widehat{y}_{1}$ or
$\widehat{y}_{q}$ are absent by convention if $m_{0}=0$.

We distinguish between interior and boundary points of $\Lambda$. Accordingly,
recalling (\ref{DISPLAY: form of LAMBDA}), there are either two or three cases
to consider according as $\gamma_{q}>0$\ or $\gamma_{q}<0$. We call these
Cases A ($\lambda\in\Lambda^{\circ}$), B$_{1}$ ($\lambda=\gamma_{1}^{-1}%
$)\ and B$_{q}$ ($\lambda=\gamma_{q}^{-1}$), this last arising when and
\textit{only} when $\gamma_{q}<0$. By convention this last condition -- that
the constraint (\ref{DISPLAY: dimn-reduced canonical form}) be, at least in
part, hyperbolic -- is implicit whenever any of the terms defined only in Case
B$_{q}$ is mentioned. In particular, this applies to $\widehat{y}_{q}$,
$\widehat{z}_{(q)}$, $\Gamma_{(q)}^{\ast}$ and$\ f_{q}$,\ defined next.

\begin{definition}\label{DEFINITION: fcf}
Let $\mathbb{P}_{\overline{\omega}^{\ast\ast}}$ be a regular dimension-reduced
canonical form and let $\lambda\in\Lambda$.

Case A $(\lambda\in\Lambda^{\circ}):$ $w(\lambda):=[I-\lambda\Delta
]^{-1}(w_{0}+\lambda d)$ denotes the unique solution to the normal equations,
while the function $f:\Lambda^{\circ}\rightarrow R$\ is defined by
$f(\lambda):=Q^{\ast}(w(\lambda))$. That is, for all $\lambda\in\Lambda
^{\circ}$:%
\begin{equation}\label{EQUATION: fcf}
f(\lambda):=%
{\textstyle\sum\nolimits_{i=1}^{q}}
(1-\lambda\gamma_{i})^{-2}\gamma_{i}\delta_{i}^{2}+2\varepsilon^{2}%
\lambda-k^{\ast}\text{.}%
\end{equation}
$\underline{f}$ and $\overline{f}$ respectively denote the infimum and
supremum\ of $\{f(\lambda):\lambda\in\Lambda^{\circ}\}$.\smallskip

Case B$_{1}$ $(\lambda=\gamma_{1}^{-1}):$ $\widehat{y}_{1}:=\varepsilon
\gamma_{1}^{-1}$ $(m_{0}>0)$, $\widehat{z}_{(1)}\in R^{q-1}$\ has general
element $[1-(\gamma_{i}/\gamma_{1})]^{-1}\delta_{i}$ $(i=2,...,q)$, while
$f_{1}:=\widehat{z}_{(1)}^{\prime}\Gamma_{(1)}^{\ast}\widehat{z}%
_{(1)}+2\varepsilon\widehat{y}_{1}-k^{\ast}$ where $\Gamma^{\ast}%
\equiv\operatorname*{diag}(\gamma_{1},\Gamma_{(1)}^{\ast})$.\smallskip

Case B$_{q}$ $(\lambda=\gamma_{q}^{-1}):$ when and \textit{only} when
$\gamma_{q}<0$: $\widehat{y}_{q}:=\varepsilon\gamma_{q}^{-1}$ $(m_{0}>0)$,
$\widehat{z}_{(q)}\in R^{q-1}$\ has general element $[1-(\gamma_{i}/\gamma
_{q})]^{-1}\delta_{i}$ $(i=1,...,q-1)$, while $f_{q}:=\widehat{z}%
_{(q)}^{\prime}\Gamma_{(q)}^{\ast}\widehat{z}_{(q)}+2\varepsilon
\widehat{y}_{q}-k^{\ast}$ where $\Gamma^{\ast}\equiv\operatorname*{diag}%
(\Gamma_{(q),}^{\ast}\gamma_{q})$.
\end{definition}

\begin{remark}
The constraint $f(\lambda)=0$\ is \textit{polynomial} in $\lambda$.
\end{remark}
Key properties of $f:\Lambda^{\circ}\rightarrow R$\ are summarised in

\begin{proposition}
\label{RESULT: key properties of f -- esp. f_ and f^_}For any $\mathbb{P}%
_{\overline{\omega}^{\ast\ast}}:$\bigskip\newline if $\delta=0_{q}$ and
$\varepsilon=0$, $f(\lambda)=-k^{\ast}$ for all $\lambda\in\Lambda^{\circ}$,
so that $\underline{f}=-k^{\ast}=\overline{f}$; \medskip\newline else, if
either $\delta$\ or $\varepsilon$\ does not vanish, $f$ is continuous and
strictly increasing, $\overline{f}>\underline{f}$ being given by$:$%
\[%
\begin{tabular}[c]{rcllrrl}%
$\overline{f}=$ & \negthinspace\negthinspace\negthinspace\negthinspace
\negthinspace$f_{1}$ & \negthinspace\negthinspace\negthinspace\negthinspace
\negthinspace$\text{or }+\infty$ & \negthinspace\negthinspace\negthinspace
\negthinspace\negthinspace$\text{according as}$ & \negthinspace\negthinspace
\negthinspace\negthinspace\negthinspace$\delta_{1}=0\text{ or}$ &
\negthinspace\negthinspace\negthinspace\negthinspace\negthinspace$\delta
_{1}>0\text{,}$ & \negthinspace\negthinspace\negthinspace\negthinspace
$\text{while}$\\
$\underline{f}=$ & \negthinspace\negthinspace\negthinspace\negthinspace
\negthinspace$f_{q}$ & \negthinspace\negthinspace\negthinspace\negthinspace
\negthinspace$\text{or }-\infty$ & \negthinspace\negthinspace\negthinspace
\negthinspace\negthinspace$\text{according as}$ & \negthinspace\negthinspace
\negthinspace\negthinspace\negthinspace$\delta_{q}=0\text{\ or}$ &
\negthinspace\negthinspace\negthinspace\negthinspace\negthinspace$\delta
_{q}>0\text{,}$ & \negthinspace\negthinspace\negthinspace\negthinspace
$(\gamma_{q}<0),$\\
$\underline{f}=$ & \negthinspace\negthinspace\negthinspace\negthinspace
\negthinspace$-k^{\ast}$ & \negthinspace\negthinspace\negthinspace
\negthinspace\negthinspace$\text{or }-\infty$ & \negthinspace\negthinspace
\negthinspace\negthinspace\negthinspace$\text{according as}$ & \negthinspace
\negthinspace\negthinspace\negthinspace\negthinspace$\varepsilon=0\text{\ or}$
& \negthinspace\negthinspace\negthinspace\negthinspace\negthinspace
$\varepsilon>0\text{,}$ & \negthinspace\negthinspace\negthinspace
\negthinspace$(\gamma_{q}>0).$%
\end{tabular}
\]

\end{proposition}
\noindent Consider now Case A. Regularity of $\mathbb{P}_{\overline{\omega
}^{\ast\ast}}$ and Proposition
\ref{RESULT: key properties of f -- esp. f_ and f^_}\ give at once

\begin{lemma}
\label{RESULT: charc'n of when w(.) does not depend upon lambda}$w(\lambda)$
does not depend upon $\lambda\Leftrightarrow\lbrack\delta=0_{q}$ and
$\varepsilon=0]\Leftrightarrow w(\lambda)=0_{q}$ for every $\lambda\in
\Lambda^{\circ}\Leftrightarrow\underline{f}=-k^{\ast}=\overline{f}$.
\end{lemma}
\noindent Again, using Lemma
\ref{RESULT: P_omega**: suff't. conditions to solve}, we have at once

\begin{lemma}
\label{RESULT: Case A: form of W(lambda)}Let $\lambda\in\Lambda^{\circ}$.
Then,
\[
W(\lambda)=\left\{
\begin{tabular}
[c]{cl}%
$\{w(\lambda)\}$ & \negthinspace\negthinspace\negthinspace if $f(\lambda)=0$\\
$\emptyset$ & \negthinspace\negthinspace\negthinspace if $f(\lambda)\neq0$%
\end{tabular}
\right.
\]
so that: $f(\lambda_{1})=f(\lambda_{2})=0\Rightarrow W(\lambda_{1}%
)=W(\lambda_{2})=\widehat{W}\Rightarrow w(\lambda_{1})=w(\lambda_{2}). $
\end{lemma}
\noindent In view of Lemma \ref{RESULT: Case A: form of W(lambda)}, we may
define $\widehat{w}_{\circ}$\ as the common value of $w(\lambda)$\ among all
solutions to $f(\lambda)=0$, when at least one such exists. Putting $W_{\circ
}:=\cup_{\lambda\in\Lambda^{\circ}}W(\lambda)$, Lemmas
\ref{RESULT: P_omega**: suff't. conditions to solve}\ and
\ref{RESULT: Case A: form of W(lambda)} now give

\begin{lemma}
\label{RESULT: W_o}$W_{\circ}=\left\{
\begin{array}
[c]{ll}%
\widehat{W}=\{\widehat{w}_{\circ}\} & \text{if }\exists\;\lambda\text{\ with
}f(\lambda)=0\\
\emptyset & \text{else.}%
\end{array}
\right.  $
\end{lemma}
\noindent Combining the above auxiliary results makes $\widehat{w}_{\circ}%
$\ explicit and, with it, the following summary of Case A.

\begin{proposition}
\label{RESULT: W_0: when nonempty? what content?}For any $\mathbb{P}%
_{\overline{\omega}^{\ast\ast}}$, there are three possibilities$:$%
\smallskip\newline(a) If $\underline{f}<0<\overline{f}$, $f(\lambda)=0$ has a
unique solution $\widehat{\lambda}$ and $\widehat{W}=W_{\circ}%
=\{w(\widehat{\lambda})\}$.\smallskip\newline(b) If $\underline{f}%
=0=\overline{f}$, $f(\lambda)=0$ for every $\lambda\in\Lambda^{\circ}$ and
$\widehat{W}=W_{\circ}=\{0_{q}\}$.\smallskip\newline(c) In all other cases,
$f(\lambda)=0$ has no solutions and $W_{\circ}=\emptyset$.
\end{proposition}

\begin{remark}
When it exists, computing $\widehat{\lambda}$ is straightforward
(see Albers et al. \cite{AlbersJMVA1,AlbersJMVA2}).
\end{remark}
In view of Proposition
\ref{RESULT: W_0: when nonempty? what content?}, we refer to $W_{\circ}$\ as a
\textit{potential Lagrangian solution set}, this potential being realised if
and only if it is non-empty. Turning now to the boundary cases, there are two
other potential Lagrangian solution sets $W_{1}:=W(\gamma_{1}^{-1})$ and
$W_{q}:=W(\gamma_{q}^{-1})$, this second definition being made when and
\textit{only} when $\gamma_{q}<0$. The hessian $H$\ being singular here, the
normal equations (\ref{DISPLAY: normal equations}) have either no solution, or
a unique solution for all but one member of $w$, which they leave
unconstrained. We have

\begin{proposition}
\label{RESULT: W_1 and W_q: when nonempty? what content?}For any
$\mathbb{P}_{\overline{\omega}^{\ast\ast}}:$\bigskip\newline(1) Case B$_{1}$
$(\lambda=\gamma_{1}^{-1}):$ $W_{1}\neq\emptyset\Leftrightarrow\lbrack
\delta_{1}=0$ and $f_{1}\leq0]$, in which case:%
\[
\widehat{W}=W_{1}=\{(\widehat{y}_{1},z_{1},\widehat{z}_{(1)}^{\,\prime
})^{\prime}:z_{1}^{2}=(-f_{1})/\gamma_{1}\}\text{.}%
\]
(2) Case B$_{q}$ $(\lambda=\gamma_{q}^{-1}<0):W_{q}\neq\emptyset
\Leftrightarrow\lbrack\delta_{q}=0$ and $f_{q}\geq0]$, in which case:%
\[
\widehat{W}=W_{q}=\{(\widehat{y}_{q},\widehat{z}_{(q)}^{\,\prime}%
,z_{q})^{\prime}:z_{q}^{2}=f_{q}/(-\gamma_{q})\}.
\]

\end{proposition}

\begin{proof}
$W_{N}(\gamma_{1}^{-1})\neq\emptyset\Leftrightarrow\delta_{1}=0$ in which
case:
\[
w\in W_{N}(\gamma_{1}^{-1})\Leftrightarrow\lbrack y=\widehat{y}_{1}\text{ and
}z_{(1)}=\widehat{z}_{(1)}]\text{, where }z\equiv(z_{1},z_{(1)}^{\prime
})^{\prime}\text{,}%
\]
$z_{1}$ being unconstrained. Thus,%
\[
W_{1}\neq\emptyset\Leftrightarrow\lbrack\delta_{1}=0\text{ and }\exists
\;z_{1}\text{\ with }\gamma_{1}z_{1}^{2}+f_{1}=0]\Leftrightarrow\lbrack
\delta_{1}=0\text{ and }f_{1}\leq0]\text{.}%
\]
(1) now follows from Lemma
\ref{RESULT: P_omega**: suff't. conditions to solve}. The proof of (2) is
entirely similar.
\end{proof}

In summary, $\mathbb{P}_{\overline{\omega}^{\ast\ast}}$ has either two
$(\gamma_{q}>0)$ or three $(\gamma_{q}<0)$ types of potential Lagrangian
solution set -- $W_{\circ}$, $W_{1}$\ and $W_{q}$ -- Propositions
\ref{RESULT: key properties of f -- esp. f_ and f^_} to
\ref{RESULT: W_1 and W_q: when nonempty? what content?} together establishing
precisely when these potentials are realised and, in each case, what the
corresponding solution set $\widehat{W}$\ then is.

\subsection{The minimised objective function and the solution set \label{SECTION: min'd objective fn and the soln set}}

We are now ready to solve any regular dimension-reduced canonical form
$\mathbb{P}_{\overline{\omega}^{\ast\ast}}$ in which, by definition, either
$m_{0}=0$\ or $[m_{0}>0$ and $\varepsilon>0]$.

To aid geometric interpretation, recall that, under the assumptions and
notation of Definition \ref{DEFINITION: canonical form}:\newline(a) $\varepsilon:=l_{0}\geq0$ denotes the
length of the orthogonal projection of $b$ onto the null space of $B$, whose
dimension is $m_{0}\geq0 $; \newline(b) for each distinct nonzero eigenvalue
$\gamma_{1}>...>\gamma_{q}$ of $B$, $\delta\in R^{q}$\ has general element
$\delta_{i}:=l_{i}/\left\vert \gamma_{i}\right\vert $ in which $l_{i}\geq0$
denotes the length of the orthogonal projection of $b$ onto the corresponding
eigenspace of $B$, whose dimension is $m_{i}\geq1$.
Accordingly:\smallskip%

\begin{tabular}
[c]{llrl}%
the linear part of the constraint vanishes & $\Leftrightarrow$ &
$m_{0}=0$ & $\Leftrightarrow\varepsilon=0,$\\
the origin is feasible & $\Leftrightarrow$ & $k^{\ast}=0,$
& \\
the origin is the target & $\Leftrightarrow$ & $\delta=0,$
& \ and:\\
the constraint is, at least in part, elliptic & $\Leftrightarrow$ &
$\gamma_{q}>0.$ &
\end{tabular}
\smallskip

We distinguish three mutually exclusive and exhaustive types of regular
dimension-reduced canonical form $\mathbb{P}_{\overline{\omega}^{\ast\ast}}$.
The first two are trivial.

We call $\mathbb{P}_{\overline{\omega}^{\ast\ast}}$ \textit{non-Lagrangian} if
$[m_{0}=0,\,k^{\ast}=0,\,\delta\neq0_{q}$ and $\gamma_{q}>0]$, (so that, in
particular, $W_{q}$ is undefined). Geometrically, if the feasible set is the
origin, the target being a positive distance\ away. As Theorem
\ref{RESULT: MAIN final theorem} establishes, in this case, both $W_{\circ}$
and $W_{1}$ are empty. That is, the constraint and normal equations are
inconsistent, so that $\mathbb{P}_{\overline{\omega}^{\ast\ast}}$\ is
\textit{not} amenable to Lagrangian solution.

We call $\mathbb{P}_{\overline{\omega}^{\ast\ast}}$
\textit{multiply-Lagrangian} if $\underline{f}=0=\overline{f}$. That is, if
$[m_{0}=0,\,k^{\ast}=0$ and $\delta=0_{q}]$. Geometrically, if the target is
the origin, through which the conic defining the constraint passes. As Theorem
\ref{RESULT: MAIN final theorem} establishes, in this case, $\widehat{W}%
$\ coincides with \textit{each} of the nonempty sets $W_{\circ}$, $W_{1}$ and,
when defined, $W_{q}$.

Finally, we call $\mathbb{P}_{\overline{\omega}^{\ast\ast}}$
\textit{singly-Lagrangian} if it is neither non-Lagrangian\ nor
multiply-Lagrangian. Algebraically, if either $[m_{0}=0,\,k^{\ast}%
=0,\,\delta\neq0_{q}$ and $\gamma_{q}<0]$,\ or $[m_{0}=0$ and $k^{\ast}\neq
0]$,\ or $[m_{0}>0$ and $\varepsilon>0]$. As Theorem
\ref{RESULT: MAIN final theorem} establishes, in this case, exactly
\textit{one} of $W_{\circ}$, $W_{1}$ and $W_{q}$\ is nonempty, and so provides
the solution set required.

\begin{theorem}
\label{RESULT: MAIN final theorem}The minimised objective function
$\underline{L}^{\ast}$ and solution set $\widehat{W}$ for a regular
dimension-reduced canonical form $\mathbb{P}_{\overline{\omega}^{\ast\ast}}%
$\ are as follows.
\begin{enumerate}
\item[(a)] If $\mathbb{P}_{\overline{\omega
}^{\ast\ast}}$\ is non-Lagrangian, $W_{\circ}=W_{1}=\emptyset$, while$:$%
\[
\underline{L}^{\ast}=%
{\textstyle\sum\nolimits_{i=1}^{q}}
\gamma_{i}^{-2}l_{i}^{2}>0\text{ attained on\ }\widehat{W}=W=\{0_{q}\}\text{.}%
\]
\item[(b)] If $\mathbb{P}_{\overline{\omega}^{\ast\ast}}$\ is multiply-Lagrangian$:$%
\[
\underline{L}^{\ast}=0\text{ attained on\ }\widehat{W}=W_{\circ}=W_{1}%
=W_{q}=\{0_{q}\}\text{.}%
\]
\item[(c)] Otherwise, if $\mathbb{P}_{\overline{\omega}^{\ast\ast}}$\ is
singly-Lagrangian, $\widehat{W}$\ is the unique nonempty member of
$\{W_{\circ},W_{1},W_{q}\}$.
Specifically, $\underline{f}$ and $\overline{f}$
being as in Proposition \ref{RESULT: key properties of f -- esp. f_ and f^_}%
$:$\smallskip\newline%
\begin{tabular}
[c]{ll}%
if $\gamma_{q}>0$, & \negthinspace\negthinspace\negthinspace\negthinspace
\negthinspace$\underline{f}<0$, while $\widehat{W}$ is $W_{\circ}$\ or $W_{1}%
$\ according as $\overline{f}>0$ or $\overline{f}\leq0;$\smallskip\\
if $\gamma_{q}<0$, & \negthinspace\negthinspace\negthinspace\negthinspace
\negthinspace$\widehat{W}$ is $W_{\circ}$, $W_{1}$\ or $W_{q} $\ according as
$\underline{f}<0<\overline{f}$, $\overline{f}\leq0$\ or $\underline{f}\geq
0.$\smallskip
\end{tabular}
\newline When $\widehat{W}=W_{\circ}$,
\[
\underline{L}^{\ast}=\widehat{\lambda}^{2}\{l_{0}^{2}+%
{\textstyle\sum\nolimits_{i=1}^{q}}
(1-\widehat{\lambda}\gamma_{i})^{-2}l_{i}^{2}\},
\]
attained at $w=w(\widehat{\lambda})$, where $\widehat{\lambda}$ uniquely solves
$f(\lambda)=0$.\medskip\newline When $\widehat{W}=W_{1}$,
\[
\underline{L}^{\ast}=(\gamma_{1}^{-1})^{2}\{l_{0}^{2}+%
{\textstyle\sum\nolimits_{i=2}^{q}}
(1-\gamma_{1}^{-1}\gamma_{i})^{-2}l_{i}^{2}\}+\zeta_{1}^{2},
\]
attained at $w=(\widehat{y}_{1},\,\pm\zeta_{1},\,\widehat{z}_{(1)}^{\,\prime
})^{\prime}$, where $\zeta_{1}:=\sqrt{(-f_{1})/\gamma_{1}}\geq0$%
.\medskip\newline When $\widehat{W}=W_{q}$,
\[
\underline{L}^{\ast}=(\gamma_{q}^{-1})^{2}\{l_{0}^{2}+%
{\textstyle\sum\nolimits_{i=1}^{q-1}}
(1-\gamma_{q}^{-1}\gamma_{i})^{-2}l_{i}^{2}\}+\zeta_{q}^{2},
\]
attained at $w=(\widehat{y}_{q},\,\widehat{z}_{(q)}^{\,\prime},\,\pm\zeta
_{q})^{\prime}$, where $\zeta_{q}:=\sqrt{f_{q}/(-\gamma_{q})}\geq0$.
\end{enumerate}
\end{theorem}

\begin{proof}
The result follows from detailed, straightforward application of the auxiliary
results of Section \ref{SECTION: Feasible solutions to the normal equations},
noting the following:
\begin{enumerate}
\item[(a)]Here, $W_{\circ}=W_{1}=\emptyset
$\ as $\underline{f}=0$ and $f_{1}>0$ while, by definition, $W=\{0_{q}%
\}$.
\item[(b)] Here, $w_{0}=\widehat{w}_{\circ}=0_{q}$ and $\widehat{z}%
_{(1)}=\widehat{z}_{(q)}=0_{q-1}$, while $f_{1}=f_{q}=0$.
\item[(c)] Here, if
$\gamma_{q}>0$, either $\varepsilon>0$ so that $\underline{f}=-\infty$, or
$m_{0}=0\neq k^{\ast}$ so that $\underline{f}=-k^{\ast}$, consistency of the
constraint implying $k^{\ast}\geq0$. In either case, $\underline{f}<0$. Again,
$\underline{f}=0=\overline{f}$\ is impossible, so that $W_{\circ}\neq
\emptyset\Leftrightarrow\underline{f}<0<\overline{f}$ when, $f$\ being
continuous and strictly increasing, $f(\lambda)=0$\ has a unique solution.
\end{enumerate} 
\end{proof}

\begin{corollary}
\label{RESULT: (non) uniqueness of soln to (d-r) c.f.}Let $\mathbb{P}%
_{\omega^{\ast}}$\ be a canonical form whose dimension-reduced form
$\mathbb{P}_{\overline{\omega}^{\ast\ast}}$\ is regular. Then:
\begin{enumerate}
\item[(1)] $\mathbb{P}_{\omega^{\ast}}$\ has a unique solution whenever
$\mathbb{P}_{\overline{\omega}^{\ast\ast}}$\ does.
\item[(2)] $\mathbb{P}_{\overline{\omega}^{\ast\ast}}$\ does \textbf{not} have a unique
solution if, and only if, it is singly-Lagrangian and either:
\begin{enumerate} 
\item[(a)] $[f_{1}<0,\delta_{1}=0]$, when its solutions are unique up to the
sign of $\widehat{z}_{1}\neq0$;  or: 
\item[(b)]
$[f_{q}>0,\delta_{q}=0]$, when its solutions are unique up to the sign of
$\widehat{z}_{q}\neq0$,
\end{enumerate}
in which cases solutions to
$\mathbb{P}_{\omega^{\ast}}$\ are unique up to orthogonal indeterminacy of (a)
$\widehat{x}_{1}$\ or (b) $\widehat{x}_{q}$\ respectively.
\end{enumerate}
\end{corollary}

\begin{proof}
The characterisation of non-uniqueness of solution to $\mathbb{P}%
_{\overline{\omega}^{\ast\ast}}$\ is immediate from Theorem
\ref{RESULT: MAIN final theorem}\ and Proposition
\ref{RESULT: key properties of f -- esp. f_ and f^_}, the rest of (2) then
following from Corollary
\ref{RESULT: charc'n of when d-r c.f. solution determines multiple c.f. solns}%
. Inspection of Theorem \ref{RESULT: MAIN final theorem}\ also establishes
that, whenever $\mathbb{P}_{\overline{\omega}^{\ast\ast}}$\ has a unique
solution, $\delta_{i}=0\Rightarrow\widehat{z}_{i}=0$. Part (1) now follows,
using again Corollary
\ref{RESULT: charc'n of when d-r c.f. solution determines multiple c.f. solns}.
\end{proof}

Theorem \ref{RESULT: MAIN final theorem}\ and its Corollary
\ref{RESULT: (non) uniqueness of soln to (d-r) c.f.}\ complete our primary
objective, providing the minimised objective function and solution set of any
regular dimension-reduced $\mathbb{P}_{\overline{\omega}^{\ast\ast}}$\ and so,
via Theorem \ref{RESULT: canonical form -- optimised}, of any initial
canonical form $\mathbb{P}_{\omega^{\ast}}$.

Two worked examples of this overall approach are given in \cite{AlbersCritchleyGowerExamples}.

\section{Intrinsic instability \label{SECTION: Intrinsic instability}}

The above analysis of any, equality or inequality constrained, problem
$\mathbb{P}_{\omega}$ rests on several partitions of possibilities. Passage
between different members of a partition can involve movement between equality
and inequality of the same two reals. Or, again, between weak ($\leq$) and strict ($<$) versions of the same
inequality. As a result, both the form of the solution set and, indeed, of the
problem itself can be intrinsically unstable under arbitrarily small
perturbations of problem parameters. 
Whether or not the solution set $\widehat{X}_{\omega}$\ varies continuously
with $\omega$ across such partition boundaries is implicit in the above
analysis. In particular, in its summative Theorems \ref{RESULT: inequality constrained variant},
\ref{Theorem 2},
\ref{RESULT: canonical form -- optimised} and
\ref{RESULT: MAIN final theorem}, and
their key Corollaries \ref{RESULT: charc'n of when d-r c.f. solution determines multiple c.f. solns} and
\ref{RESULT: (non) uniqueness of soln to (d-r) c.f.}. The discontinuities involved can be
dramatic, as the following instances illustrate.

Potential synergetic uses of these analytical insights in connection with efficient numerical optimisation methods are 
noted in the closing discussion (Section \ref{SECTION: Worked example}).

\subsection{Instability of the form of the solution set}\label{SUBSECTION:91}

We begin with a marked instance of the passage from non-unique to unique solutions. As in Example
\ref{EXAMPLE: min dist centre to sphere}, consider minimisation of the distance to a sphere from its centre, whose
solution set is, of course, the sphere itself. Changing this to minimising the
same distance from \textit{any} other point, the solution suddenly becomes
unique. This instability corresponds to the passage from $\delta_{1}=0$\ to
$\delta_{1}>0$\ in Corollary
\ref{RESULT: (non) uniqueness of soln to (d-r) c.f.}.  

In the hyperbolic case, there can be extreme directional instability, due to
the orthogonality of the eigenspaces of $\gamma_{1}$ and $\gamma_{q}$. For
instance, consider minimisation of the distance from the origin to the
hyperbola $z_{1}^{2}-z_{2}^{2}=k^{\ast}$, both of whose eigenvalues, $\pm1$,
are simple. When $k^{\ast}=0$, this comprises the lines $z_{2}=\pm z_{1}$
which meet at the (repeated) solution $(0,0)^{\prime}$, while $\mathbb{P}%
_{\overline{\omega}^{\ast\ast}}$\ is multiply-Lagrangian. Otherwise, it
comprises two branches with these lines as asymptotes, $\mathbb{P}%
_{\overline{\omega}^{\ast\ast}}$\ being singly-Lagrangian. For $k^{\ast}>0$,
these branches meet the $z_{1}$--axis at the twin solutions $\pm(\sqrt
{k^{\ast}},0)^{\prime}$ while, for $k^{\ast}<0$, they meet the $z_{2}$--axis
at the twin solutions $\pm(0,\sqrt{-k^{\ast}})^{\prime}$. Thus, no matter how
small we take $k_{+}^{\ast}>0>k_{-}^{\ast}$, the twin solutions for
$k_{+}^{\ast}$ are always orthogonal to those for $k_{-}^{\ast}$. In terms of
Theorem \ref{RESULT: MAIN final theorem}, the solution set $\widehat{W}%
$\ changes from $W_{1}$\ to $W_{q}$ via $W_{\circ}$ as $k^{\ast} $\ goes from
positive to negative via zero.  

Geometrically, it is clear that this same extreme directional instability can
arise in the elliptic case. Consider, for example, minimising the distance
from the origin to the ellipse $z_{1}^{2}+z_{2}^{2}/\kappa^{2}=1$ $(\kappa>0)$
so that, when $\kappa=1$, the solution set is the ellipse itself. Then, with
$\kappa_{+}>1>\kappa_{-}$, no matter how close we take $\kappa_{+}$ and
$\kappa_{-}$\ to $1$ -- and so, to each other -- the solution set for
$\kappa_{+}$ is $\pm(1,0)^{\prime}$ while, for $\kappa_{-}$, it remains in the
directions $\pm(0,1)^{\prime}$ orthogonal to these.

\subsection{Instability of the form of the problem}
Intrinsic instability of the form of the problem comes, itself, in a variety
of forms, as we illustrate.

As a first instance,  
consider the variant of Example \ref{EXAMPLE 51}(b) with constraint $(x_{1}+\xi_{1}%
)x_{0}=1$ for given $\xi_{1}\neq0$. Whereas the perfect solution here has
$\widehat{x}_{1}=0$ for every $\xi_{1}$, it has $\widehat{x}_{0}=1/\xi_{1}$
which tends to $+\infty$ as $\xi_{1}\rightarrow0_{+}$ and to $-\infty$ as
$\xi_{1}\rightarrow0_{-}$. Such instability holds generally. Indeed, Theorem
\ref{Theorem 2} shows that, if $\mathbb{P}_{\omega}$ admits an essentially
perfect solution, it is intrinsically unstable under small constraint
perturbations. And, relatedly, that the occurrence of essentially perfect
solutions to $\mathbb{P}_{\omega}$ is itself intrinsically unstable, depending
as it does on the vanishing of the constraint parameter $c_{0}$. Replacing
$c_{0}$ by $C_{10}$, the same is true of reducibility of $\mathbb{P}_{\omega}$
when $B_{00}$ is singular. That is, when $A$ is positive semi-definite, there
is intrinsic instability at the boundaries between the three possibilities --
$\mathbb{P}_{\omega}$ admits a perfect solution,\ $\mathbb{P}_{\omega}$ admits
an essentially perfect solution, or $\mathbb{P}_{\omega}$ admits a projected,
yet imperfect, reduced form -- their occurrences being characterised in
Theorem \ref{Theorem 2}.

As a second instance, any positive
semi-definite $A$ is arbitrarily close to a positive definite matrix $A(\kappa)$, 
$\kappa>0$  -- for example,
$A(\kappa)=A+\kappa I$ -- analysis of the problem changing from that of Section
\ref{SECTION: A p.s.d.} to that discussed in the sequel. 

Similar intrinsic instabilities in problem form occur at the boundaries
between members of the following partitions made for positive definite
$A$:
\begin{enumerate}
\item[(i)] the four possible forms of constraint -- elliptic,
hyperbolic, or the partly parabolic variant of either -- as detailed in
Definition \ref{DEFINITION: dimn-reduced canonical form};
\item[(ii)] in the dimension-reduced canonical form
$\mathbb{P}_{\omega^{\ast\ast}}$, whether it is regular or not -- as detailed
in Definition \ref{DEFINITION 72};
\item[(iii)] the three possibilities --
non-Lagrangian, multiply Lagrangian and singly Lagrangian -- as detailed in
Section \ref{SECTION: min'd objective fn and the soln set}.
\end{enumerate}

\section{Discussion\label{SECTION: Worked example}}

The current approach to the problem addressed in this paper, and the more
generally applicable approach of generalised trust region (GTR) methods,
broadly reviewed at the outset, may be compared as follows. Their intrinsic
complementarity is evident and raises two substantial questions for future research.

Being derivative-based, the\ GTR approach is, intrinsically,\textit{\ local}.
A local optimum is identified\textit{\ }implicitly, and then
sought\ numerically, via simultaneous satisfaction of an appropriate set of
conditions. These may include conditions additional to those of the problem
itself (notably, those introduced to ensure numerical stability). Instances of
the problem not satisfying such conditions require separate resolution. Under
certain such additional constraint qualifications, it may be possible to
characterise global optimality: see, for example, \cite{PongWolkowicz,More,SternWolkowicz}.
Unless a local
optimum can be shown to be unique, identifying a global optimum requires
repeated use of one or more algorithms. Of course, at most a finite number of
local optima may be obtained in this way. The desired trade-off to be struck between the
speed, accuracy and numerical stability of the overall computational procedure
depends on context and purpose.

In contrast, the approach presented here is, intrinsically,\textit{\ global
}in three distinct senses: (a) it partitions all possible instances of the
problem, and of its solution set, without exception; (b) its simplifying
transformations are applied to the whole space; and: (c) its solution sets
comprise global optima only. These sets are specified\ explicitly in closed
algebraic form, affording both insight and interpretation. In particular, it
is shown that, depending on precisely defined circumstances, there may be 0,
1, 2, or more -- indeed, uncountably many -- global optima, geometry
illuminating when each possibility occurs. Overall, the current global,
analytic approach throws a distinctive algebraic and geometric spotlight on
the diverse nature of different instances of both problem and solution set;
and, on the possibility of \textit{intrinsic} instability of both, pinpointing
precisely when this occurs: specifically, when passage between its different
partition members involve distinctions that cannot be drawn in the
floating-point world, yet lead to radical changes in the form of either the
solution set, or of problem itself.

Two substantial questions arise at once:

\begin{enumerate} 
\item  can the current global, analytic approach be extended to accommodate
possible non-convexity of the objective function?

\item  can the intrinsic complementarity it shares with the local,
computational GTR approach be exploited, releasing potential synergies between them?
\end{enumerate} 
Motivations for examining these questions in future work, and first pointers
towards possible answers, are briefly offered in closing.

Concerning (1), broadening the applicability of the current approach in this
way would bring it closer to that of GTR methods. One natural strategy here
would work with affine transformations such that the partitioned least-squares
form $\operatorname{diag}(I_{r},O_{n-r})$ of $A$ is replaced by
$\operatorname{diag}(I_{r_{+}},O_{n-r},-I_{r_{-}})$, for some affine
invariants $r_{+}\geq0$, $r_{-}\geq0$\ with $r_{+}+r_{-}=r$.

Concerning (2), generic motivation here comes from seeking to capitalise on
the distinctive advantages of both approaches: the former providing, in
principle, global solution sets that are complete, exact, insightful and
interpretable; the latter providing, in practice, algorithms whose
specification and deployment can be varied to deliver a balance of speed,
accuracy and numerical stability well-suited to a given context and purpose.
Especially strong when the present problem arises as one iteration of a
general purpose optimisation procedure, specific motivation here includes the
need to handle the possibility of various forms of intrinsic -- not just
numerical -- instability highlighted by the present global analysis and
illustrated in Section \ref{SECTION: Intrinsic instability}: among these, the possibility of undetectably small
changes leading to solutions in \textit{orthogonal} directions is particularly
striking (see Section \ref{SUBSECTION:91}). One natural strategy here would be to modify existing
computational procedures so that, when numerical checks informed by global
analysis indicate its appropriateness, multiple forms of problem or solution
set are entertained. This strategy has the potential advantage of identifying
several, quite different, locally-promising search directions at each iteration.

Concerning both, it will be of interest to see how far, and by what means,
progress can best be made in addressing these challenging, open questions.

\bibliographystyle{plain} 
\bibliography{laa_bib}
\clearpage 
\appendix
\section{Two worked examples}
We provide with two worked examples and a short conclusion. The examples differ
only by a single element of $A$, yet involve completely different approaches
to their solution. As such, they illustrate one of the intrinsic instabilities
noted in Section 9. Namely, that at the boundary between $A\succeq O$\ where
Section 5 applies, and $A\succ O$\ where Sections 6 to 8 pertain.

Whereas this generic type of instability can of course arise whatever the
scale of the problem, to make things explicit, we work here with $n=3$ and
with values of $\omega$\ for which exact calculations are fairly
straightforward. Specifically, we take $k=1$, $b=0_{3}$, $t=(1,1,1)^{\prime}$,
$B=I_{3}$\ and%
\[
A=A(\kappa):=\left(
\begin{array}
[c]{ccc}%
1 & 0 & 0\\
0 & 1 & -1\\
0 & -1 & 1+\kappa
\end{array}
\right)  ,\;\kappa\geq0,
\]
so that $A(0)\succeq O$, while $A(\kappa)\succ O$ for all $\kappa>0$. While
reporting exact results, we also indicate where there is -- or is not --
numerical uncertainty, especially concerning which member of a partition of
possibilities the problem on hand belongs to.

Examples \ref{EXAMPLE:1} and \ref{EXAMPLE:2} concern $\kappa=0$\ and $\kappa=3/2$\ respectively. In
the first case, one of the computed eigenvalues of $A$ will be within machine
accuracy of zero; in the second, all will be clearly positive. As
$\kappa\rightarrow0+$, both problem forms will be flagged up. Numerically, $B$
is clearly positive definite.

\begin{example}\label{EXAMPLE:1} 
Here, $A=A(0)$\ clearly has rank $r=2$, and we follow
the approach of Section \ref{SECTION: A p.s.d.}.

Analysis begins with a three-stage transformation: first, to centred least-squares
form, as in part (i) of Theorem 3.1; second, to simplified form, via (5)
of Lemma 5.2; and third, to simultaneous diagonal form, via the further linear
transformation of Remark 5.2. Overall, this transformation is $x\rightarrow
x_{g}=T^{-1}(x-t)$\ where
\[
T=\left(
\begin{array}
[c]{ccc}%
1 & 0 & 0\\
0 & -\frac{1}{2} & -\frac{1}{\sqrt{2}}\\
0 & \frac{1}{2} & -\frac{1}{\sqrt{2}}%
\end{array}
\right)
\]
inducing $\omega\rightarrow\omega_{g}$\ with $t_{g}=0_{3}$\ and $A_{g}%
=T^{\prime}AT=\operatorname*{diag}(1,1,0)$. These two components of
$\omega_{g}$ are known (and so do not require calculation), as is the diagonal
nature of $B$. To within machine accuracy, $B_{g}=T^{\prime}%
BT=\operatorname*{diag}(1,\tfrac{1}{2},1)$, $b_{g}=(1,0,-\sqrt{2})^{\prime}%
$\ and $k_{g}=-2$.

Next, we apply Theorem 5.1 to $\mathbb{P}_{\omega_{g}}$, dropping the
subscript $g$ where notationally convenient. Identifying terms in Definitions
\ref{DEF52} and \ref{DEF53} for this transformed member of $\Omega_{3}$, and noting that the
third diagonal element $B_{g}$ is clearly positive, we have $s_{0}=1$, so that
$B_{00}=\Gamma_{0}=(1)$, while $b_{0}=d_{0}=(-\sqrt{2})$. In particular,
$s_{0}=n-r$, so that part (ii) of Theorem \ref{Theorem 2} applies. Moreover, to within
machine accuracy, $k_{1}=0$\ and $Z_{0}(0)=\{\sqrt{2}\}$, so that
$\mathbb{P}_{\omega_{g}}$\ admits perfect solution $\widehat{x}_{g}%
=(0_{2}^{\prime},\sqrt{2})^{\prime}$, $\widehat{L}_{\omega_{g}}=0$.

Finally, invoking Theorem 2.1, we transform back to conclude that
$\mathbb{P}_{\omega}$\ has perfect solution $\widehat{x}=t+T\widehat{x}%
_{g}=(1,0,0)^{\prime}$, $\widehat{L}_{\omega}=0$.
\end{example}

\begin{example}\label{EXAMPLE:2}
Here, $A=A(3/2)$\ clearly has full rank, and we follow
Sections 6 to 8.

Analysis begins with a two-stage transformation: first, to centred -- here, full
-- least-squares form, using again Theorem 3.1(i); second, to canonical form,
via the Euclidean transformation of Theorem 6.2. Overall, apart from a
constant, this transformation is $x\rightarrow x_{g}=T^{-1}x$\ where
\[
T=\left(
\begin{array}
[c]{ccc}%
0 & -1 & 0\\
-\sqrt{\frac{8}{5}} & 0 & \frac{1}{\sqrt{15}}\\
-\sqrt{\frac{2}{5}} & 0 & -\frac{2}{\sqrt{15}}%
\end{array}
\right)
\]
inducing $\omega\rightarrow\omega_{g}$ in which $A_{g}=T^{\prime}AT$ is
$I_{3}$ and $B_{g}=T^{\prime}BT$ is diagonal. To within machine accuracy,
$B_{g}=\operatorname*{diag}(2,1,\tfrac{1}{3})$, which clearly has distinct,
positive, eigenvalues. Accordingly, $\mathbb{P}_{\omega_{g}}$\ is, itself,
dimension-reduced and regular. Finally, in the notation of Definition 6.1, we
compute that $\delta=(3/\sqrt{10},1,\sqrt{3/5})^{\prime}$ and $k^{\ast}=1$.

We solve the regular dimension-reduced canonical form $\mathbb{P}_{\omega_{g}%
}$\ via the steps described in Section 8. Since, clearly, $m_{0}=0$\ and
$k^{\ast}\neq0$, $\mathbb{P}_{\omega_{g}}$\ is singly-Lagrangian. Here,
$\Lambda^{\circ}=(-\infty,\gamma_{1}^{-1})=(-\infty,\tfrac{1}{2})$, the
function $f:\Lambda^{\circ}\rightarrow R$\ as given in Definition 8.2 taking the form:%
\begin{equation} \label{EQUATION: fcf example}
\frac{9}{5(1-2\lambda)^{2}}+\frac{1}{(1-\lambda)^{2}}+\frac{9}{5(3-\lambda
)^{2}}-1.
\end{equation}
Again, it is numerically clear that $\gamma_{3}>0,$ $\varepsilon=0$\ and
$\delta_{1}>0$\ so that, by Proposition 8.1, $\underline{f}=-k^{\ast}=-1<0$ and
$\overline{f}=+\infty$. Accordingly (Proposition 8.2), $\widehat{W}=W_{\circ
}=\{w(\widehat{\lambda})\}$\ where $\widehat{\lambda}$\ is the unique root of
$f(\lambda)=0$. Albers et al \cite{AlbersJMVA1} show how to reduce $\Lambda^{\circ}$\ to a
finite interval $(l^{\circ},u^{\circ})$, $u^{\circ}:=\gamma_{1}^{-1}$,
containing $\widehat{\lambda}$. Here, $l^{\circ}=\tfrac{1}{2}-3\sqrt{3/5}$.
This is visualised in Figure \ref{FIGURE}, where the functional form (\ref{EQUATION: fcf example}) is plotted
over an interval containing $(l^{\circ},u^{\circ})$, its two vertical lines
being at $\lambda=l^{\circ}$\ and $\lambda=u^{\circ}$, while there are
vertical asymptotes at $\lambda=\gamma_{i}^{-1}$, $(i=1,2,3)$. Numerical
solution over $(l^{\circ},u^{\circ})$ -- for example, by the bisection method
-- is straightforward, yielding here the approximate value $\widehat{\lambda
}\approx-0.527$.

Finally, substituting in the relevant part of Theorem 8.1(c), and back
transforming $\widehat{x}_{g}\rightarrow\widehat{x}$, yields the final
solution $\widehat{x}\approx(0.655,0.414,0.632)^{\prime}$, $\widehat{L}%
_{\omega}\approx0.370$.
\begin{figure}
\begin{center}
\includegraphics[width=9cm]{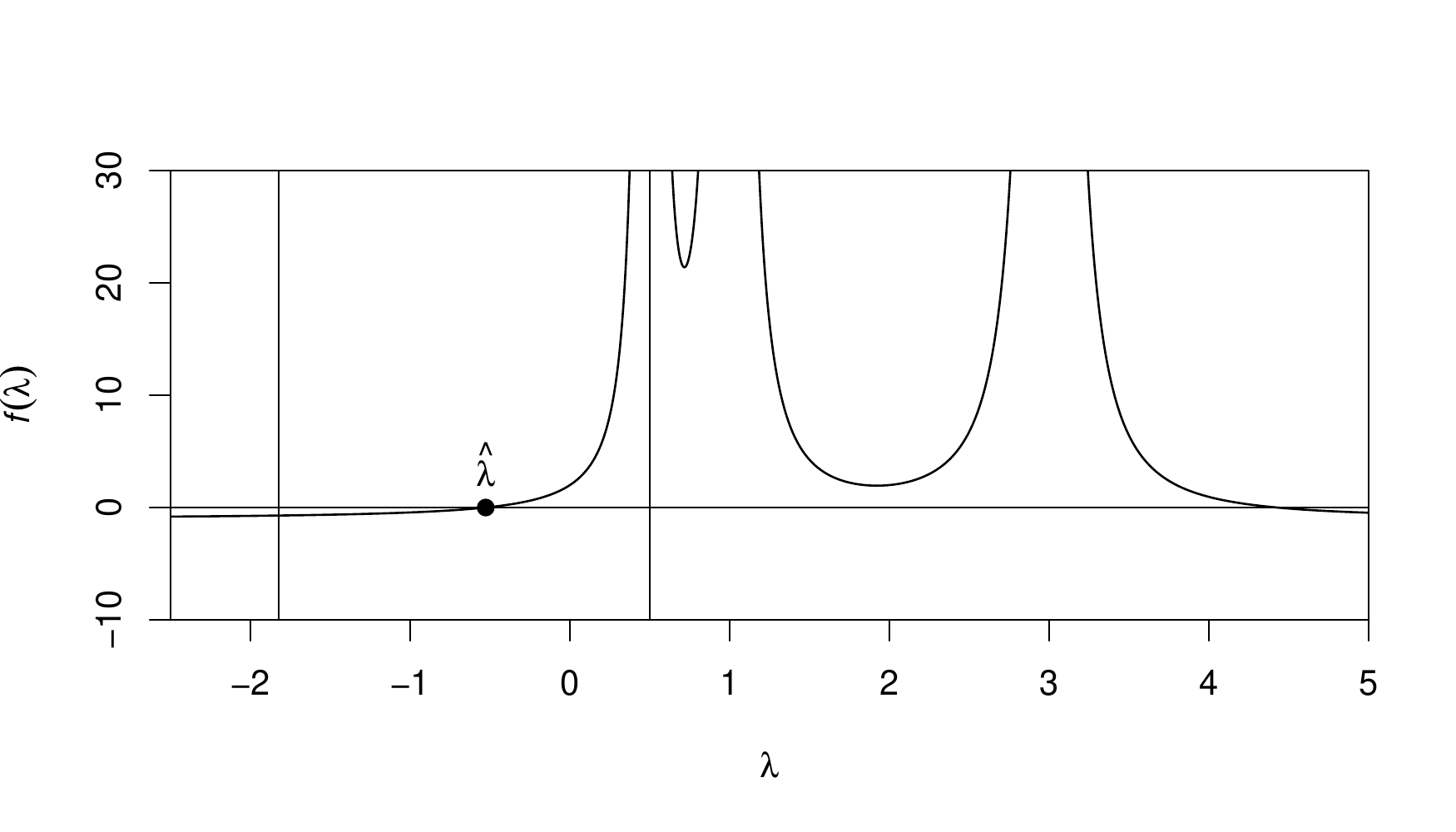}
\end{center}
\caption{Display of the functional form of $f(\lambda)$. The vertical lines are at $\lambda=l^{\circ}$\ and $\lambda=u^{\circ}$, defined in the
text.}\label{FIGURE}
\end{figure}
\end{example}

\end{document}